\documentclass[10pt,a4paper]{amsart}
\usepackage{amsmath,amssymb,color,setspace}
\usepackage{hyperref}
\usepackage[utf8,utf8x]{inputenc}
\usepackage{fullpage}

\usepackage{array}
\usepackage{epsfig}
\usepackage{xspace}
\usepackage{color}

\usepackage{enumitem}

\newtheorem{propo}{Proposition}[section]
\newtheorem{corol}[propo]{Corollary}
\newtheorem{theor}[propo]{Theorem}
\newtheorem{lemma}[propo]{Lemma}

\newtheorem*{theor*}{Theorem}

\theoremstyle{definition}
\newtheorem{defin}[propo]{Definition}

\theoremstyle{remark}
\newtheorem{remar}[propo]{Remark}

\newcommand{\NN }{\mathbb{N}}

\newcommand{\ZZ }{\mathbb{Z}}
\newcommand{\id }{\mathrm{id}}
\DeclareMathOperator{\SL}{SL}
\DeclareMathOperator{\mdop}{mod}
\newcommand{\et }{\eta}
\newcommand{\ql }{\left[}
\newcommand{\qr }{\right]}
\newcommand{\lemt }{t}
\newcommand{\md }[1]{\:\:(\mdop\: #1)}

\definecolor{darkgreen}{rgb}{0.0,0.1,0.6}

\title[Frieze patterns over finite commutative local rings]
{Frieze patterns over finite commutative local rings}

\author{Bernhard~B\"ohmler}
\address{Bernhard~B\"ohmler, Leibniz Universit\"at Hannover,
Institut f\"ur Algebra, Zah\-lentheorie und Diskrete Mathematik,
Fakult\"at f\"ur Mathematik und Physik,
Wel\-fengarten 1,
D-30167 Hannover, Germany}
\email{boehmler@math.uni-hannover.de}
\author{Michael~Cuntz}
\address{Michael Cuntz, Leibniz Universit\"at Hannover,
Institut f\"ur Algebra, Zah\-lentheorie und Diskrete Mathematik,
Fakult\"at f\"ur Mathematik und Physik,
Wel\-fengarten 1,
D-30167 Hannover, Germany}
\email{cuntz@math.uni-hannover.de}

\begin{document}


\begin{abstract}
We count numbers of tame frieze patterns with entries in a finite commutative local ring. For the ring $\mathbb{Z}/p^r\mathbb{Z}$, $p$ a prime and $r\in\mathbb{N}$ we obtain closed formulae for all heights. These may be interpreted as formulae for the numbers of certain relations in quotients of the modular group.
\end{abstract}

\maketitle

\section{Introduction}

\noindent The special linear group $\SL_2(\mathbb{Z})$ is, for example, generated by the following two matrices:
$$
S=\begin{pmatrix} 0 & -1 \\ 1 & 0 \end{pmatrix}, \quad
T=\begin{pmatrix} 1 & 1 \\ 0 & 1 \end{pmatrix}.
$$
Since $S^2=-\id$, any element in $\SL_2(\mathbb{Z})$ can be written (up to a sign) as a product of matrices of the form~$\eta(a):=T^a S$ for $a\in \mathbb{Z}$.
The relations among these matrices $\eta(a)$, $a\in \mathbb{Z}$, are particularly interesting.

For example, a sequence of positive integers $c_1,\ldots,c_n$ such that $\eta(c_1)\cdots\eta(c_n)=-\id$ is called a \emph{quiddity cycle} because it is the quiddity of a Conway-Coxeter frieze pattern. It turns out that these quiddity cycles are in bijection with triangulations of a convex $n$-gon by non-intersecting diagonals and are thus counted by Catalan numbers \cite{jChC73}.
More generally, one can count or parametrize the set of solutions to $\eta(c_1)\cdots\eta(c_n)=-\id$ with $c_1,\ldots,c_n\in\mathbb{Z}$ (cf.~\cite{CH19}),
or for example $\eta(c_1)\cdots\eta(c_n)=\pm\id$ with $c_1,\ldots,c_n\in\mathbb{N}_{>0}$ (cf.~\cite{O18}).
As in \cite{C19}, we call a solution to
$$
\eta(c_1)\cdots\eta(c_n)=\varepsilon\id
$$
an \emph{$\varepsilon$-quiddity cycle} (where $c_1,\ldots,c_n$ are elements of a ring and $\eta(\cdot)$ is as in Definition \ref{def:eta}, which is compatible with the above $\eta$).

A solution $c_1,\ldots,c_n\in \mathbb{Z}/N\mathbb{Z}$ to $\eta(c_1)\cdots\eta(c_n)=\pm\id$ may be viewed as a solution over $\mathbb{Z}$ to $\eta(c_1)\cdots\eta(c_n)\in \Gamma$ where $\Gamma$ is a congruence subgroup of $\SL_2(\mathbb{Z})$. Thus it is interesting to count the number of such solutions. With the chinese remainder theorem, this can be reduced to the case in which $N$ is a power of a prime.
Such a closed formula appeared first in \cite{MG21} for finite fields when the product is equal to $-\id$, a generalization with arbitrary matrix on the right is contained in \cite{CM24}, and the case of $\mathbb{Z}/2^r\mathbb{Z}$ of odd length is considered in \cite{M24}.
In this article we give closed formulae for these numbers of solutions over all rings $\mathbb{Z}/N\mathbb{Z}$, $N\in\mathbb{N}$, and a recursion for the case of an arbitrary matrix on the right side, i.e.\ for the solutions of $\eta(c_1)\cdots\eta(c_n)=A$ for arbitrary $A\in \SL_2(\mathbb{Z}/N\mathbb{Z})$.

In Section \ref{sec:localrings}, we first count the quiddity cycles of odd lengths in the more general case of commutative finite local rings:

\begin{theor}[Thm.\ \ref{thm:nfriezesI}]
Let $R$ be a finite commutative local ring with maximal ideal $I\unlhd R$, $\varepsilon\in \{\pm 1\}\subset R$, and $n\in\mathbb{N}_{>1}$ with $n$ odd or $(-1)^{n/2}\ne\varepsilon$.
If $\omega_n$ is the number of $\varepsilon$-quiddity cycles over $R/I$ of length $n$, then the number of $\varepsilon$-quiddity cycles over $R$ is
$$
\omega_n \cdot |I|^{n-3}.
$$
\end{theor}
As a corollary, this implies most of the previously known results for quiddity cycles over residue class rings.

In a second much more technical part (Section \ref{sec:pr}) we give closed formulae for the number of solutions of arbitrary lengths over $\mathbb{Z}/p^r\mathbb{Z}$, $p$ an arbitrary prime.
For $m\in\mathbb{N}$, $q\in \mathbb{Z}\setminus \{\pm 1\}$, we write $\binom{m}{2}_q := \frac{(q^m-1)(q^{m-1}-1)}{(q-1)(q^2-1)}$ and $[m]_q:=\frac{q^m-1}{q-1}$.

\begin{theor}[Thm.\ \ref{thm:nfriezes_even}]
Let $R=\mathbb{Z}/p^r\mathbb{Z}$ for a prime $p$ and $r\in\mathbb{N}$, $I=pR$ the maximal ideal, $n\in\mathbb{N}_{>1}$ with $n$ even, and $\varepsilon\in\{\pm 1\}$.
If $\omega_n$ is the number of $\varepsilon$-quiddity cycles over $R/I$ of length $n$, then the number of $\varepsilon$-quiddity cycles over $R$ is
$$
\begin{cases}
(\omega_n-1) \cdot p^{(r-1)(n-3)} + \sigma_n & \quad \text{if}\quad \varepsilon=(-1)^{n/2} \\
\omega_n \cdot p^{(r-1)(n-3)} & \quad \text{if}\quad \varepsilon=-(-1)^{n/2} \text{ and } p>2 \\
(\omega_n-1) \cdot p^{(r-1)(n-3)} & \quad \text{if}\quad \varepsilon=-(-1)^{n/2} \text{ and } p=2
\end{cases},
$$
where
\begin{eqnarray*}
\sigma_2 &=& 1, \quad\quad
\sigma_4=(r-1)p^r-(r-2)p^{r-1},
\\
\sigma_{n}&=&
p^{{\left(n/2 - 1\right)} r} [r-1]_{p^{n/2-2}} - p^{{\left(n/2 - 1\right)} r - 1} [r-2]_{p^{n/2-2}}.
\end{eqnarray*}
\end{theor}

Note that the numbers $\omega_n$ are known:
\begin{theor}[{\cite[Thm.\ 1]{MG21}, \cite[Thm.\ 1.1]{CM24}}]
Let $q$ be a prime power and $n\in\mathbb{N}$, $n>4$. Then the number $\omega_n$ of $\varepsilon$-quiddity cycles over $\mathbb{F}_q$ of length $n$ is
$$
\omega_n =
\begin{cases}
\left[\frac{n-1}{2}\right]_{q^{2}} & \quad \text{if}\quad n \equiv 1 \md{2}, \\
(q-1)\binom{n/2}{2}_{q}+q^{n/2-1} & \quad \text{if}\quad n \equiv 0 \md{2},\quad \varepsilon=(-1)^{n/2}, \\
(q-1)\binom{n/2}{2}_{q} & \quad \text{if}\quad n \equiv 0 \md{2}, \quad \varepsilon=-(-1)^{n/2}, \quad q \text{ odd}.
\end{cases}
$$
\end{theor}

\medskip
Finally in Section \ref{sec:recursion}, we determine recursions for the number of solutions to
$$
\eta(c_1)\cdots\eta(c_n)= A
$$
for $c_1,\ldots,c_n\in \mathbb{Z}/N\mathbb{Z}$, $N\in\mathbb{N}$ and $A$ an arbitrary matrix. A closed formula could theoretically be computed for each fixed matrix $A$.

\section{Quiddity cycles and relations}

\begin{defin}\label{def:eta}
Let $R$ be a ring and $a\in R$. Then we set
\[ \et(a):=\begin{pmatrix} a & -1 \\ 1 & 0 \end{pmatrix}. \]
For $c_1,\ldots,c_n\in R$ we write
\[ \ql c_1,\ldots,c_n \qr := \prod_{i=1}^n \et(c_i). \]
\end{defin}

\begin{lemma}\label{lem:uv1}
Let $R$ be a commutative ring, $a,b,u,v,x,y \in R$ such that $a,(uv-1)\in R^\times$. Then
\begin{eqnarray}
\label{eq:red43} \ql x,u,v,y \qr &=& \ql x+\frac{1-v}{uv-1}, uv-1, y+\frac{1-u}{uv-1} \qr, \\
\label{eq:red1} \ql x,1,y \qr &=& \ql x-1,y-1 \qr.
\end{eqnarray}
\end{lemma}
\begin{proof}
The first equation already appeared in \cite[Lemma 4.2]{CH19}; the second one is the special case when $u=1$.
\end{proof}

The following lemma was discovered by the second author 2022 when working with quiddity cycles over rings.
We shall not need it for our main theorems, but it is the key for the recursion in the last section of this article.
It is also the main ingredient for the main results of \cite{M24}; although the lemma was published there first, F.\ Mabilat acknowledged the origin of this lemma in his article.
Note that the result in \cite{M24} is a very special case of our result on local rings in Section \ref{sec:localrings} which is proven with a completely different idea.

\begin{lemma}
Let $R$ be a commutative ring, $c,u,v,b,d \in R$. Then
\begin{eqnarray}
\label{eq:red53N} \ql c,u,v,b,d \qr &=& \ql c-\frac{vb-2}{x}, x, d-\frac{uv-2}{x} \qr
\end{eqnarray}
for $$x := \frac{(vb-1)(uv-1)-1}{v}$$
if $v,x\in R^\times$.
\end{lemma}
\begin{proof}
Direct computation.
\end{proof}

\begin{defin}
For $a \in R^\times$ we write
$$
\lambda_a :=
\begin{pmatrix} a & 0 \\ 0 & a^{-1} \end{pmatrix}.
$$
\end{defin}

\begin{lemma}
\label{lem:lambda}
Let $R$ be a commutative ring, $(c_{1},\ldots,c_{n}) \in R^{n}$ and $\lemt\in R^\times$.
If $n$ is odd, then
$$
\begin{pmatrix} \lemt & 0 \\ 0 & 1 \end{pmatrix}
\ql c_1,\ldots, c_n\qr
\begin{pmatrix} 1 & 0 \\ 0 & \lemt^{-1} \end{pmatrix}
= \ql \lemt c_1, \lemt^{-1} c_2, \lemt c_3,\lemt^{-1} c_4,\ldots,\lemt c_n \qr.
$$
In particular, if $\ql c_1,\ldots, c_n\qr = \lambda_a$ for $a\in R^\times$, then
$$
\ql \lemt c_1, \lemt^{-1} c_2, \lemt c_3,\lemt^{-1} c_4,\ldots,\lemt c_n \qr
= \lambda_{t a}.
$$
\end{lemma}
\begin{proof}
Observe first that for $a,b\in R$,
$$
\begin{pmatrix} \lemt^{-1} & 0 \\ 0 & 1 \end{pmatrix}
\begin{pmatrix} \lemt a & -1 \\ 1 & 0 \end{pmatrix}
\begin{pmatrix} \lemt^{-1} b & -1 \\ 1 & 0 \end{pmatrix}
\begin{pmatrix} \lemt & 0 \\ 0 & 1 \end{pmatrix}
=
\begin{pmatrix} a & -1 \\ 1 & 0 \end{pmatrix}
\begin{pmatrix} b & -1 \\ 1 & 0 \end{pmatrix},
$$
and
$$
\begin{pmatrix} \lemt & 0 \\ 0 & 1 \end{pmatrix}
\begin{pmatrix} a & -1 \\ 1 & 0 \end{pmatrix}
\begin{pmatrix} 1 & 0 \\ 0 & \lemt^{-1} \end{pmatrix}
=
\begin{pmatrix} \lemt a & -1 \\ 1 & 0 \end{pmatrix}.
$$
Thus
\begin{eqnarray*}
\ql c_1,\ldots, c_n\qr
\begin{pmatrix} 1 & 0 \\ 0 & \lemt^{-1} \end{pmatrix}
&=&
\begin{pmatrix} \lemt^{-1} & 0 \\ 0 & 1 \end{pmatrix}
\ql \lemt c_1, \lemt^{-1} c_2 ,\ldots,\lemt^{-1} c_{n-1} \qr
\begin{pmatrix} \lemt & 0 \\ 0 & 1 \end{pmatrix} \eta(c_n)
\begin{pmatrix} 1 & 0 \\ 0 & \lemt^{-1} \end{pmatrix}
\\
&=&
\begin{pmatrix} \lemt^{-1} & 0 \\ 0 & 1 \end{pmatrix}
\ql \lemt c_1, \lemt^{-1} c_2,\ldots,\lemt^{-1} c_{n-1} \qr
\eta(\lemt c_n)
\end{eqnarray*}
which is the claim.
\end{proof}

\section{Quiddity cycles over finite local rings}\label{sec:localrings}

\begin{defin}[{\cite[Def.\ 2.2]{C19}}]
Let $R$ be a commutative ring, $\underline{c}=(c_1,\ldots,c_n)\in R^n$, and $\varepsilon\in R$. Then $\underline{c}$ is called an \emph{$\varepsilon$-quiddity cycle} if
$$
\ql c_1,\ldots,c_n \qr = \varepsilon \id = \lambda_\varepsilon.
$$
Note that since this product is in $\SL_2(R)$, we always have $\varepsilon=\varepsilon^{-1}$.
For instance, if $R$ is a domain, then~$\varepsilon\in \{\pm 1\}$.
\end{defin}

\subsection{Odd length}

\begin{theor}\label{thm:nfriezesI}
Let $R$ be a finite commutative local ring with maximal ideal $I\unlhd R$, $\varepsilon\in \{\pm 1\}\subset R$, and~$n\in\mathbb{N}_{>1}$ with $n$ odd or $(-1)^{n/2}\ne\varepsilon$.
If $\omega_n$ is the number of $\varepsilon$-quiddity cycles over $R/I$ of length $n$, then the number of $\varepsilon$-quiddity cycles over $R$ is
$$
\omega_n \cdot |I|^{n-3}.
$$
\end{theor}
\begin{proof}
The condition of being a quiddity cycle is a polynomial identity.
Thus if $\kappa : R \longrightarrow R/I$ is the canonical projection and $(c_1,\ldots,c_n)$ is an $\varepsilon$-quiddity cycle, then
$(\kappa(c_1),\ldots,\kappa(c_n))$ is an $\kappa(\varepsilon)$-quiddity cycle over $R/I$.
For each fixed $\varepsilon$-quiddity cycle $\underline{c}=(\overline{c_1},\ldots,\overline{c_n})\in (R/I)^n$, we count the number of $\varepsilon$-quiddity cycles
which project to this cycle under $\kappa$.

Since $I$ is maximal, $R/I$ is a field. Assume first that $\underline{c}$ contains no unit, i.e.\ $\underline{c}=(0,\ldots,0)$. Then
\begin{equation}\label{eq:00n}
\ql 0,\ldots,0 \qr =
\begin{cases}
\begin{pmatrix}
0 & (-1)^n \\ -(-1)^n & 0
\end{pmatrix} & \text{ if } n\equiv 1 \md{2}, \\
\begin{pmatrix}
(-1)^{n/2} & 0 \\ 0 & (-1)^{n/2}
\end{pmatrix} & \text{ if } n\equiv 0 \md{2},
\end{cases}
\end{equation}
so $n$ is even and $\varepsilon=(-1)^{n/2}$. These are exactly the cases that are excluded by assumption.
Thus we may assume that there is an entry in $\underline{c}$ which is a unit, after rotating the cycle, without loss of generality $\overline{c_2}\in (R/I)^\times$.

Let $(c_1,\ldots,c_n)\in R^n$ be an $\varepsilon$-quiddity cycle that maps to $\underline{c}$ under $\kappa$; we have $c_2\in R^\times$ because $c_2\notin I$.
If $\ql c_1,\ldots,c_n \qr = \varepsilon \id$, then there exist $f,g\in R$, $h\in R^\times$ such that
$$
\begin{pmatrix}
(1+fg)/h & f \\
g & h
\end{pmatrix}
= \varepsilon \ql c_4,\ldots,c_n \qr^{-1} =
\eta(c_1)\eta(c_2)\eta(c_3)=
\begin{pmatrix}
c_1 c_2 c_3 -c_1 -c_3 & 1-c_1 c_2 \\
c_2 c_3 -1 & -c_2
\end{pmatrix}.
$$
Thus the entries $c_1,c_2,c_3$ are uniquely determined by $f,g,h$:
$$
c_2=-h,\quad c_1=(1-f)/c_2,\quad c_3=(1+g)/c_2.
$$
Since there are $|I|^{n-3}$ possible choices\footnote{Note that any such choice implies $h\in R^\times$ since $\overline{h}=-\overline{c_2}\ne 0$.}
for $c_4,\ldots,c_n$ that map to $\overline{c_4},\ldots,\overline{c_n}$,
we obtain $\omega_n \cdot |I|^{n-3}$ different $\varepsilon$-quiddity cycles of length $n$.
\end{proof}

As an example, we use the theorem to obtain closed formulae for $(-1)$-quiddity cycles:

\begin{corol}
Let $n\in\mathbb{N}_{>1}$ with $n\not\equiv 2 \md{4}$, $p$ be a prime, and $r\in\NN_{>0}$.
Then the numbers of sequences $(c_1,\ldots,c_n)\in \mathbb{Z}/p^r\mathbb{Z}$ such that $\ql c_1,\ldots,c_n \qr = -\id$ are as follows.
\begin{enumerate}
\item[(a)] For $n\equiv 1 \md{2}$ we have\quad $p^{(n-3)(r-1)} \left[\frac{n-1}{2}\right]_{p^2}$\quad sequences.
\item[(b)] For $n\equiv 0 \md{4}$ we have\quad $p^{(n-3)(r-1)}(p^{\frac{n-2}{2}}-1)\left[\frac{n}{4}\right]_{p^2}$\quad sequences if $p>2$.
\end{enumerate}
\end{corol}
\begin{proof}
This is Theorem \ref{thm:nfriezesI} for $R=\mathbb{Z}/p^r\mathbb{Z}$ and $I=(p)$.
Note that for $p=2$ and $n$ even, the number $(-1)^{n/2}$ is always equal to $\varepsilon=1$, so in this case the Theorem does not apply.
\end{proof}

\begin{remar}
In a finite local ring $R$, every element is either a unit or nilpotent.
Indeed, let $a\in R$. The set $\{a^i \mid i\in \mathbb{N}\}$ is finite, so there exist $0\le m<n$ with minimal $m$ such that $a^m=a^n$.
This implies $a^m (a^{n-m}-1) = 0$.
If $m=0$, then $a$ is a unit; otherwise, $a^m=0$ since $a^{n-m}-1$ is a unit.
On the other hand, if $a$ is nilpotent, say $a^n=0$, then $(1+a) \sum_{i=0}^{n-1} (-a)^i=1$.
\end{remar}

\subsection{Even length}

In order to understand the missing cases (even length), we need to understand those cycles in which all entries lie in the maximal ideal.

\begin{defin}
Let $R$ be a finite commutative local ring with maximal ideal $I\unlhd R$.
For $n\ge 2$ and $z\in R$, denote
\begin{equation}
\sigma_{z,n} := \{ (c_1,\ldots,c_n)\in I^n \mid  \ql c_1,\ldots,c_n\qr = \lambda_z \}.
\end{equation}
\end{defin}

\begin{lemma}
If $z\in R$ and $n\ge 2$, then $\sigma_{z,n}\ne \emptyset$ implies that $n$ is even and
$z\in (-1)^{n/2} + I$.
\end{lemma}
\begin{proof}
If $(c_1,\ldots,c_n)\in I^n$ such that $\ql c_1,\ldots,c_n\qr = \lambda_z$, then in $R/I$,
$\ql 0,\ldots,0\qr = \lambda_{z+I}$. But the computation of $\ql 0,\ldots,0\qr$ in the proof of Theorem \ref{thm:nfriezesI}, Equation \ref{eq:00n} shows that $z+I = (-1)^{n/2}+I$ and $n\equiv 0 \md{2}$.
\end{proof}

\begin{theor}\label{thm:2nfriezesI}
Let $R$ be a finite commutative local ring with maximal ideal $I\unlhd R$,
$\varepsilon\in \{\pm 1\}\subseteq R$,
$z\in \varepsilon+I \subseteq R^\times$,
and $n\in\mathbb{N}_{>3}$ with $n$ even.
Then
$$
|\sigma_{z,n}| = \sum_{x\in -\varepsilon+I} |\sigma_{x,n-2}| \cdot
|\{ (u,v)\in I \times I \mid uv-1 = x/z \}|.
$$
\end{theor}
\begin{proof}
Let $(c,u,v,b,\ldots)\in I^n$ be a sequence with $\ql c,u,v,b,\ldots\qr = \lambda_z$.
Since $x:=uv-1$ is a unit, using Equation (\ref{eq:red43}) from Lemma \ref{lem:uv1} we can reduce the sequence to a sequence of length $n-1$ via
\begin{eqnarray*}
\lambda_z &=& \ql c,u,v,b,\ldots\qr = \ql c+(1-v)/(uv-1), uv-1, b+(1-u)/(uv-1),\ldots\qr \\
&=& \ql c+(1-v)/x, x, b+(1-u)/x,\ldots\qr.
\end{eqnarray*}
Since $n-1$ is odd, we may apply Lemma \ref{lem:lambda} and Equation (\ref{eq:red1}) from Lemma \ref{lem:uv1}:
$$
\lambda_{xz} = \ql cx+(1-v), 1, bx+(1-u),\ldots\qr
= \ql cx-v, bx-u,\ldots\qr.
$$
Thus the resulting sequence $(cx-v, bx-u,\ldots)$ is contained in $\sigma_{xz,n-2}$ because it consists of elements of $I$; note that $xz\in -\varepsilon+I$.
We obtain a map
$$\rho : \sigma_{z,n} \longrightarrow \bigcup_{x\in -\varepsilon+I
} \sigma_{x,n-2}.$$

Conversely, take a sequence $(e,f,\ldots)$ in $\sigma_{x,n-2}$ for some
$x\in -\varepsilon+I$.
Then by Equation (\ref{eq:red1}) from Lemma \ref{lem:uv1},
$$
\lambda_x = \ql e,f,\ldots \qr = \ql e+1,1,f+1,\ldots \qr.
$$
With Lemma \ref{lem:lambda} and $y:=x/z \in -1+I$,
$$
\lambda_z = \ql (e+1)/y,y,(f+1)/y,\ldots \qr.
$$
Now if $uv-1 = y = x/z$ for $u,v\in I$, then by Lemma \ref{lem:uv1}
\begin{eqnarray*}
\lambda_z &=& \ql (e+1)/(uv-1), uv-1, (f+1)/(uv-1), \ldots \qr \\
&=& \ql (e+v)/(uv-1), u, v, (f+u)/(uv-1), \ldots \qr.
\end{eqnarray*}
This last sequence is in $\sigma_{z,n}$; thus the above map $\rho$ is surjective.
The preimages of $\rho$ are parametrized by $\sigma_{x,n-2}$, $x\in R^\times$ and $u,v\in I$ with $uv-1 = x/z$. Moreover, the entries with labels $3,\ldots,n-2$ in the parameters ensure that the preimages are unique; $\rho$ is injective.
\end{proof}

\section{Quiddity cycles modulo a power of a prime}
\label{sec:pr}

We now concentrate on the case of the local ring $R:=\ZZ/p^r\ZZ$ for $p$ a prime and $r\in\mathbb{N}$.
The maximal ideal is $I=pR$.
Write $\nu_p(a)$ for the largest $k$ such that $a\equiv 0 \md{p^k}$; we agree that $\nu_p(0)=r$.

\begin{lemma}\label{lem:mu_uv}
For $a\in R\setminus R^\times$, we let $M:=\{(u,v) \in I\times I \mid uv = a\}$. Then,
$$
|M| =
\begin{cases} (r-1)p^r-(r-2)p^{r-1} & \text{ if } a = 0, \\
               (p^r-p^{r-1})(\nu_p(a)-1) & \text{ if } a\ne 0.
\end{cases}
$$
\end{lemma}

\begin{proof}

Let $a=0$. Note that $uv\equiv 0\pmod{p^r}$ if and only if $\nu_p(u)+\nu_p(v)\geq r$. Consider the ring $R$ as an ordered set. Only each $p$-th element is divisible by $p$. Hence, precisely $|R|-\frac{|R|}{p}$ elements are divisible by $1=p^0$ but not by $p$. Moreover, precisely $\frac{|R|}{p} - \frac{|R|}{p^2}$ elements are divisible by $p$ but not by $p^2$. Continuing this way, we deduce that, for $1\leq i\leq r-1$, 
$$|\{x\in I\ |\ \nu_p(x)=i\}|=p^{r-i}-p^{r-i-1},$$
as $|R|=p^r$. Therefore,
$$|\{(u,v)\in M\ |\ \nu_p(u)=i\}|=(p^{r-i}-p^{r-i-1})\cdot p^i,$$
due to the following argument. There are exactly $\widetilde{n}:= (p^i-p^{i-1})+(p^{i-1}-p^{i-2})+\ldots $ elements $v$ in $I$ with $\nu_p(v)\geq r-i$ wherefore we obtain that $\widetilde{n}=p^i$, as $\nu_p(0)=r$. Furthermore, if $u=0$, there are $p^{r-1}$ possible choices for $v$ such that $(u,v)\in M$. This yields
$$
|M|=1\cdot p^{r-1} + \sum\limits_{i=1}^{r-1} {(p^{r-i} - p^{r-i-1})\cdot p^i} = p^{r-1} + \sum\limits_{i=1}^{r-1} {p^{r-i}\cdot (1-\frac{1}{p})\cdot p^i} = p^{r-1} + (r-1)\cdot p^r\cdot (1-\frac{1}{p}).$$
The assertion follows. Next, let $a\neq 0$. Note that $\nu_p(a)\geq 1$. We write
$$a=a_1\cdot p^{\nu_p(a)},\quad u=u_1\cdot p^{\nu_p(u)},\quad v=v_1\cdot p^{\nu_p(v)},$$
where $(u_1,p)=1=(v_1,p)=(a_1,p)$. Then, the equation $uv=a$ becomes
$$uv=u_1 v_1 p^{\nu_p(u) + \nu_p(v)} =a_1p^{\nu_p(a)}.$$
We can express $\nu_p(a)$ in exactly $\nu_p(a)-1$ ways as the sum of the two positive integers $\nu_p(u)$ and $\nu_p(v)$. In each such case, we may fix $u_1$ and obtain $v_1 = u_1^{-1}\cdot a_1$ as a possible solution of the equation $uv=a$. If $\nu_p(u)=i$ then there are exactly $p^{r-i} - p^{r-i-1}$ different choices for $u$ and hence for $u_1$.
To our particular solution $v_1=v_{\textup{particular}}$ chosen above we have to add all possibilities for $v_1$ which yield a total valuation of $uv$ which is larger than or equal to $r$, that is, all other numbers $\tilde{v}_1$ such that $\nu_p(\tilde{v}_1) \geq r-i$. In total, there are~$p^i$ such numbers, as $\nu_p(0)=r$.

\noindent Altogether, there are
$$|M|= \sum\limits_{i=1}^{\nu_p(a)-1} {(p^{r-i} - p^{r-i-1})\cdot p^i} = (p^r-p^{r-1}) (\nu_p(a)-1)$$
solutions in this case.
\end{proof}

\begin{lemma}\label{lem:sigma_const}
For even $n$ and $z,z'\in R$,
$$
\nu_p(z-(-1)^{n/2}) = \nu_p(z'-(-1)^{n/2}) \quad \Longrightarrow \quad
|\sigma_{z,n}|=|\sigma_{z',n}|.
$$
\end{lemma}

\begin{proof}
We proceed by induction over $n$.
If $n=2$, then since
$$
\ql c_1, c_2 \qr = \begin{pmatrix}
c_1 c_2-1 & -c_1 \\ c_2 & -1
\end{pmatrix},
$$
we get that if $\ql c_1,c_2\qr = \lambda_z$ for some $z$, then $c_1=c_2=0$ and $z=-1$.
Thus $|\sigma_{z,2}|=0$ for all $z$ with $\nu_p(z-(-1)^{n/2})>0$ and $|\sigma_{-1,2}|=1$.

Recall from Theorem \ref{thm:2nfriezesI}
that for $z\in \varepsilon+I \subseteq R^\times$,
and even $n\in\mathbb{N}_{>3}$ we have
$$
|\sigma_{z,n}| = \sum_{x\in -\varepsilon+I} |\sigma_{x,n-2}| \cdot
|\{ (u,v)\in I \times I \mid uv-1 = x/z \}|.
$$
We fix $x:=-\varepsilon +bp^k$ and let $z=\varepsilon + \tilde{a}p^j$. Then, $z^{-1} = \varepsilon - \tilde{a}p^j +\ldots - \ldots = \varepsilon + ap^j$ for some $a$ which is coprime to $p$. We inspect the different cases:
first, assume $k\neq j$. We have 
$$1+xz^{-1}= 1+(-\varepsilon + bp^k)(\varepsilon +ap^j) = b\varepsilon p^k -a\varepsilon p^j +abp^{k+j}.$$
Hence, $\nu_p(1+xz^{-1})=\min\{\nu_p(x+\varepsilon), \nu_p(z- \varepsilon)\}$ if $\nu_p(x+\varepsilon)\neq \nu_p(z-\varepsilon)$.
Next, assume $\nu_p(x+\varepsilon) = \nu_p(z-\varepsilon)=k$. In this case, we may write $z^{-1}=\varepsilon +ap^k$. We take another element $z^\prime\in \varepsilon +I$ with the property that ${z^\prime}^{-1} = \varepsilon + a^\prime p^k$. We let $x^\prime :=- \varepsilon +b^\prime p^k$ where $b^\prime$ is yet to be specified. Then, $\nu_p(xz^{-1}+1)= k+\nu_p(\varepsilon (b-a)+abp^k)$ and $\nu_p(x^\prime{z^\prime}^{-1}+1)= k+\nu_p(\varepsilon (b^\prime - a^\prime)+a^\prime b^\prime p^k)$. Now, we choose
$$b^\prime := a^\prime a^{-1}z^\prime z^{-1} b.$$
Hence, 
$$\frac{a}{a^\prime}(\varepsilon (b^\prime - a^\prime)+ a^\prime b^\prime p^k) = \varepsilon (b - a)+ab p^k$$
and therefore
$$\nu_p(\varepsilon (b^\prime - a^\prime)+ a^\prime b^\prime p^k) = \nu_p(\varepsilon (b - a)+ab p^k).$$
Next, we set $\mathcal{N}_k := \{w\in -\varepsilon + I\ |\ \nu_p(w+\varepsilon) = k\}$ and define the map $\varphi_{z, z^\prime}: \mathcal{N}_k \rightarrow \mathcal{N}_k$ by
$$\varphi_{z,z^\prime} (x) = \varphi_{z,z^\prime} (-\varepsilon +bp^k) := -\varepsilon +b^\prime p^k = -\varepsilon + (a^\prime a^{-1}z^\prime z^{-1} b)p^k = x^\prime.$$
As an inverse is given by
$$\varphi_{z,z^\prime}^{-1} (x^\prime) = \varphi_{z,z^\prime}^{-1} (-\varepsilon +b^\prime p^k) :=
-\varepsilon +(z{z^\prime}^{-1}a{a^\prime}^{-1}b^\prime)p^k,$$
the map $\varphi_{z, z^\prime}$ is bijective.
\end{proof}

\begin{defin}
In view of Lemma \ref{lem:mu_uv} and Lemma \ref{lem:sigma_const}, we will write
$$
\mu(j) := \begin{cases} (r-1)p^r-(r-2)p^{r-1} & j = r, \\
             (p^r-p^{r-1})(j-1) & j<r
\end{cases}
$$
and $\sigma_n(\ell) := |\sigma_{(-1)^{n/2}+p^\ell,n}|$. Moreover, let
$$n_j := |\{a\in R \mid \nu_p(a)=j \}| =
\begin{cases} 1 & j=r \\
p^{r-j}-p^{r-j-1} & j<r
\end{cases}
$$
for $j=0,\ldots,r$.
\end{defin}

\begin{theor}\label{thm:recusion_for_sigma_n}
The numbers $\sigma_{n}(\ell)$, $n\in\mathbb{N}$, $\ell=1,\ldots,r-1$ satisfy:
\begin{eqnarray*}
\sigma_n(\ell)
&=& \sum_{j=1}^{\ell-1} n_j \sigma_{n-2}(j)\mu(j) + \sum_{j=\ell+1}^{r} n_j \sigma_{n-2}(j)\mu(\ell)\\
&& + \sigma_{n-2}(\ell) \left(\mu(\ell)(p-2)p^{r-\ell-1} + \sum_{j=\ell+1}^{r} n_j \mu(j)\right), \\
\sigma_n(r) &=& \sum_{j=1}^r n_j \sigma_{n-2}(j)\mu(j).
\end{eqnarray*}
\end{theor}
\begin{proof}
With the new notation, Theorem \ref{thm:2nfriezesI} translates to
\begin{eqnarray*}
\sigma_n(\ell) &=& \sum_{x\in -\varepsilon+I} \sigma_{n-2}(\nu_p(x+\varepsilon)) \cdot
\mu(\nu_p(x/(\varepsilon+p^\ell)+1)) \\
&=& \sum_{x\in I} \sigma_{n-2}(\nu_p(x)) \cdot \mu(\nu_p((-\varepsilon+x)\cdot (\varepsilon+p^\ell)+1)) \\
&=& \sum_{x\in I} \sigma_{n-2}(\nu_p(x)) \cdot \mu(\nu_p(-\varepsilon p^\ell + \varepsilon x + x p^\ell)) \\
&=& \sum_{x\in I} \sigma_{n-2}(\nu_p(x/(\varepsilon+p^\ell))) \cdot \mu(\nu_p(-\varepsilon p^\ell + x)) \\
&=& \sum_{x\in I} \sigma_{n-2}(\nu_p(x)) \cdot \mu(\nu_p(x-\varepsilon p^\ell))
\end{eqnarray*}
where $\varepsilon=(-1)^{n/2}$.
In the last sum we distinguish the cases $\nu_p(x)<\ell$, $\nu_p(x)>\ell$, $\nu_p(x)=\ell$.
Note that the case $\ell=r$ has to be treated separately (see below), so assume $\ell<r$ first.
If $\nu_p(x)<\ell$, then $\nu_p(x-\varepsilon p^\ell)=\nu_p(x)$; if $\nu_p(x)>\ell$, then $\nu_p(x-\varepsilon p^\ell)=\ell$:
\begin{eqnarray*}
\sigma_n(\ell)
&=& \sum_{j=1}^{\ell-1} n_j \sigma_{n-2}(j)\mu(j) + \sum_{j=\ell+1}^{r} n_j \sigma_{n-2}(j)\mu(\ell)\\
&& +
\sigma_{n-2}(\ell) \underbrace{\sum_{y=1}^{p-1} \sum_{w=0}^{p^{r-\ell-1}-1} \mu(\ell + \nu_p(pw + y -\varepsilon))}_{x = (y+pw)p^\ell, \:\: \nu_p(x)=\ell}.
\end{eqnarray*}
In the third sum we consider the cases $y\ne\varepsilon$ and $y=\varepsilon$:
\begin{eqnarray*}
\sigma_n(\ell)
&=& \sum_{j=1}^{\ell-1} n_j \sigma_{n-2}(j)\mu(j) + \sum_{j=\ell+1}^{r} n_j \sigma_{n-2}(j)\mu(\ell)\\
&& +
\sigma_{n-2}(\ell) \left(\underbrace{\mu(\ell)(p-2)p^{r-\ell-1}}_{y\ne\varepsilon} + \underbrace{\sum_{w=0}^{p^{r-\ell-1}-1} \mu(\ell + 1 + \nu_p(w))}_{y=\varepsilon}\right) \\
&=& \sum_{j=1}^{\ell-1} n_j \sigma_{n-2}(j)\mu(j) + \sum_{j=\ell+1}^{r} n_j \sigma_{n-2}(j)\mu(\ell)\\
&& +
\sigma_{n-2}(\ell) \left(\mu(\ell)(p-2)p^{r-\ell-1} + \sum_{j=\ell+1}^{r} n_j \mu(j)\right).
\end{eqnarray*}
If $\ell=r$, we obtain the simpler formula
\begin{eqnarray*}
\sigma_n(r) &=& \sum_{j=1}^r n_j \sigma_{n-2}(j)\mu(j)
\end{eqnarray*}
because then $\nu_p(x-\varepsilon p^\ell)=\nu_p(x)$.
\end{proof}

\begin{lemma}\label{lem:some_sums}
Let $p$ be a prime number, let $r\in\mathbb{N}_{\geq 2}$ and let $\ell \in\{1,\ldots, r-1\}$. Then, the following assertions hold:
\begin{enumerate}
\item[\rm (a)] $\sum\limits_{i=1}^{m} p^{-i} =\frac{1-p^{-m}}{p-1}$,
\item[\rm (b)] $\sum\limits_{j=1}^{m} jp^{-j} =\frac{p^{-m}\left((-1-m)p+m+p^{m+1}\right)}{{(p-1)}^2}$,
\item[\rm (c)] $\sum\limits_{j=\ell +1}^{r-1} (p^{r-j}-p^{r-j-1})\cdot (p^r-p^{r-1})\cdot (j-
1) =\\ p^{r-1}(-rp+r-1)-p^{2r-2-\ell}(-p-\ell p+\ell ) -p^{2r-2}(-p^{2-r}+p^{1-r}) +p^{2r-2}(-p^{1-\ell}+p^{-\ell})$.
\end{enumerate}
\end{lemma}

\begin{proof}
 \begin{enumerate}
\item[(a)] Direct verification using the geometric sum formula.
\item[(b)] For $x\in\mathbb{R}\setminus \{1\}$ we have
$$\sum\limits_{j=1}^m -jx^{-j-1}=\frac{d}{dx}\left( \sum\limits_{j=1}^m x^{-j} \right) =\frac{d}{dx}\left( \frac{1-x^{-m}}{x-1} \right).$$ Therefore,
$$\sum\limits_{j=1}^m jx^{-j}=(-x)\cdot \frac{d}{dx}\left( \frac{1-x^{-m}}{x-1} \right) = \frac{x^{-m}\left((-1-m)x+m+x^{m+1}\right)}{{(x-1)}^2}.$$
The claim follows.
\item[(c)] We have
\begin{eqnarray*}
&{\phantom{=}}& \sum\limits_{j=\ell +1}^{r-1} (p^{r-j}-p^{r-j-1})\cdot (p^r-p^{r-1})\cdot (j-
1) ={(p^r-p^{r-1})}^2 \sum\limits_{j=\ell +1}^{r-1} p^{-j}(j-1)\\
&=& {(p^r-p^{r-1})}^2 \left( \left(\sum\limits_{j=1}^{r-1} jp^{-j} - \sum\limits_{j=1}^{\ell} jp^{-j}\right) - \left(\sum\limits_{j=1}^{r-1} p^{-j} - \sum\limits_{j=1}^{\ell} p^{-j} \right) \right)\\
&\stackrel{(a), (b)}{=}& {(p^r-p^{r-1})}^2 \left(\frac{p^{-(r-1)}(r-rp-1+p^r)}{{(p-1)}^2} - \frac{p^{-\ell}((-1-\ell)p+\ell+p^{\ell +1})}{{(p-1)}^2} - \frac{1-p^{1-r}}{p-1} +\frac{1-p^{-\ell}}{p-1}\right)\\
&=&{(p^{r-1})}^2[p^{-(r-1)}(-rp+r-1+p^r) - p^{-\ell}((-1-\ell)p+\ell+p^{\ell +1})\\
&-&(p-1)(1-p^{1-r})+(p-1)(1-p^{-\ell})]\\
&=&{(p^{r-1})}^2[p^{-(r-1)}(-rp+r-1+p^r) - p^{-\ell}((-1-\ell)p+\ell+p^{\ell +1})+(p-1)(p^{1-r}-p^{-\ell})]\\
&=&{(p^{r-1})}^2[p^{-(r-1)}(-rp+r-1) - p^{-\ell}((-1-\ell)p+\ell)+(p-1)(p^{1-r}-p^{-\ell})]\\
&=&p^{r-1}(-rp+r-1)-p^{2r-2-\ell}(-p-\ell p+\ell ) -p^{2r-2}(-p^{2-r}+p^{1-r}) +p^{2r-2}(-p^{1-\ell}+p^{-\ell})
\end{eqnarray*}
as claimed.
\end{enumerate}
\end{proof}

\begin{lemma}\label{lem:nice_expression_for_parts_of_sigma_n}
For $r\in\mathbb{N}_{\geq 2}$ and $\ell=1,\ldots,r-1$ we have:
$$
\mu(\ell)(p-2)p^{r-\ell-1} + \sum_{j=\ell+1}^{r} n_j \mu(j) =
(\ell-1)p^{2r-\ell}-(2\ell-3)p^{2r-\ell-1}+(\ell-1)p^{2r-\ell-2}.
$$
\end{lemma}
\begin{proof}
Using the definitions of $\mu$ and $n_j$ we obtain that our claim is equivalent to
\begin{eqnarray*}
&\phantom{=}&
(p^r-p^{r-1})\cdot (\ell -1)\cdot (p-2)\cdot p^{r-\ell -1}\\
&+& \sum\limits_{j=\ell +1}^{r-1} (p^{r-j}-p^{r-j-1}) (p^r-p^{r-1}) (j-1) + 1\cdot ((r-1)p^r -(r-2)p^{r-1})\\
&=& (\ell-1)p^{2r-\ell}-(2\ell-3)p^{2r-\ell-1}+(\ell-1)p^{2r-\ell-2}.
\end{eqnarray*}
By Lemma \ref{lem:some_sums}(c) the left-hand side of the last equation is equal to 
\begin{align*}
&{\phantom{\Leftrightarrow}}
(p^r-p^{r-1})\cdot (\ell -1)\cdot (p-2)\cdot p^{r-\ell -1}\\
&+p^{r-1}(-rp+r-1)-p^{2r-2-\ell}(-p-\ell p+\ell ) -p^{2r-2}(-p^{2-r}+p^{1-r}) +p^{2r-2}(-p^{1-\ell}+p^{-\ell})\\
&+ 1\cdot (r-1)p^r -(r-2)p^{r-1}\\
&= p^{2r-2-\ell}[(p-1)(\ell -1)(p-2)+(p+\ell p -\ell)-p+1]\\
&+p^{r-1}(-rp+r-1) -p^{2r-2}(-p^{2-r}+p^{1-r})   + 1\cdot (r-1)p^r -(r-2)p^{r-1}\\
&= p^{2r-2-\ell}[(\ell -1)(p^2-3p+2)+\ell p -\ell +1)]\\
&+ rp^{r-1} -rp^r -p^{r-1} +p^r -p^{r-1} +rp^r -p^r -rp^{r-1}+2p^{r-1}\\
&= p^{2r-2-\ell}[p^2\ell -3p\ell +2\ell -p^2 +3p -2 +p\ell -\ell +1]\\
&= p^{2r-2-\ell}(p^2\ell -p^2 -2p\ell +\ell +3p -1)\\
&= (\ell-1)p^{2r-\ell}-(2\ell-3)p^{2r-\ell-1}+(\ell-1)p^{2r-\ell-2}.
\end{align*}
\end{proof}

\begin{lemma}\label{lem:formula_for_sum_not_only_involving_geometric_things_but_simultaneously_also_involving_q_integers}
    Let $p$ be a prime number, let $s\in \mathbb{N}_{\geq 2}$ and let $m\in \mathbb{N}_{\geq 3}$. Then, the following holds.
    \begin{align*}
    &{\phantom{0}}\sum\limits_{j=2}^{s-1}p^{-j(m-1)}[j-1]_{p^{m-2}}(j-1)\\
    &= \frac{1}{p^{m-2}-1}\left( \frac{(1-s)p^{4-s-m}+(s-2)p^{3-s-m}+p^{2-m}}{{(p-1)}^2} - \frac{(1-s)p^{(m-1)(2-s)} + (s-2)p^{(m-1)(1-s)}+1}{{(p^{m-1}-1)}^2} \right)\!.
    \end{align*} 
\end{lemma}

\begin{proof}
    We have
\begin{align*}
&{\phantom{0}}\sum\limits_{j=2}^{s-1}p^{-j(m-1)}[j-1]_{p^{m-2}}(j-1) = \sum\limits_{j=2}^{s-1} p^{-j(m-1)}\frac{p^{(m-2)(j-1)}-1}{p^{m-2}-1} (j-1)\\
&= \frac{1}{p^{m-2}-1}\sum\limits_{j=2}^{s-1} (p^{2-m-j}-p^{-j(m-1)})(j-1)\\
&= \frac{p^{2-m}}{p^{m-2}-1}\left( \sum\limits_{j=2}^{s-1} jp^{-j} - \sum\limits_{j=2}^{s-1} p^{-j}\right) - \frac{1}{p^{m-2}-1}\left( \sum\limits_{j=2}^{s-1} j{(p^{m-1})}^{-j} -        \sum\limits_{j=2}^{s-1} {(p^{m-1})}^{-j}\right)\\
&= \frac{p^{2-m}}{p^{m-2}-1}\cdot \left( \frac{p^{-s+1}(-sp+s-1+p^s)}{{(p-1)}^2} -\frac{1}{p} - \frac{1-p^{-s+1}}{p-1} + \frac{1}{p} \right)\\
&- \frac{1}{p^{m-2}-1}\left( \frac{p^{(m-1)(1-s)}(-sp^{m-1}+s-1+p^{(m-1)s})}{{(p^{m-1}-1)}^2} - \frac{1}{p^{m-1}} - \frac{1-p^{(m-1)(1-s)}}{p^{m-1}-1} +\frac{1}{p^{m-1}} \right)\\
&=\frac{p^{2-m}}{p^{m-2}-1}\cdot \frac{p^{1-s}(s-sp-1+p^s)+(1-p)(1-p^{1-s})}{{(p-1)}^2}\\
&- \frac{1}{p^{m-2}-1}\cdot \frac{p^{(m-1)(1-s)}(-sp^{m-1}+s-1+p^{(m-1)s}) + (1-p^{m-1})(1-p^{(m-1)(1-s)})}{{(p^{m-1}-1)}^2}\\
&= \frac{p^{2-m}}{p^{m-2}-1}\cdot \frac{-sp^{2-s}+sp^{1-s}-p^{1-s}+1-p^{1-s}+p^{2-s}}{{(p-1)}^2}\\
&- \frac{1}{p^{m-2}-1}\cdot \frac{-sp^{(m-1)(2-s)}+(s-1)p^{(m-1)(1-s)}+1-p^{(m-1)(1-s)}+p^{(m-1)(2-s)}}{{(p^{m-1}-1)}^2}\\
&= \frac{1}{p^{m-2}-1}\left( \frac{(1-s)p^{4-s-m}+(s-2)p^{3-s-m}+p^{2-m}}{{(p-1)}^2} - \frac{(1-s)p^{(m-1)(2-s)} + (s-2)p^{(m-1)(1-s)}+1}{{(p^{m-1}-1)}^2} \right)\!.
\end{align*}
\end{proof}

\begin{defin}
We use the notation of $q$-numbers,
$$
[n]_q := \sum_{i=0}^{n-1} q^i = \frac{q^n-1}{q-1}
$$
where the last equality requires $q\ne 1$.
\end{defin}

\begin{lemma}\label{lem:polynomial_long_division}
Let $x\in \mathbb{R}^{+}\setminus \{1\}$ and let $r\in\mathbb{N}$. Then, the following equation holds:
$$\frac{(1-r)x^{2-r}+(r-2)x^{1-r}+1}{x-1} =(1-r)x^{1-r} + x^{1-r}[r-1]_x.$$
\end{lemma}

\begin{proof}
By polynomial long division, we obtain
\begin{align*}
\left((1-r)x^{2-r}+(r-2)x^{1-r}+1\right)/(x-1) &= (1-r)x^{1-r} - \frac{x^{1-r}-1}{x-1} = (1-r)x^{1-r} -\frac{\frac{1-x^{r-1}}{x^{r-1}}}{x-1}\\
&=(1-r)x^{1-r} + \frac{1}{x^{r-1}}\cdot \frac{x^{r-1}-1}{x-1}=(1-r)x^{1-r} + x^{1-r}[r-1]_x.
\end{align*}
\end{proof}

\begin{lemma}\label{lem:alpha_beta_gamma}
Let $p$ be a prime number, let $r\in\mathbb{N}_{\geq 3}$, and let $x\in \mathbb{R}^{+}\setminus \{1\}$. Set
\begin{align*}
\alpha &:= p^rx^r(x^{r-1}-1)(\frac{x}{p}-1)-p^{r-1}x^r(x^{r-2}-1)(\frac{x}{p}-1),\\
\beta &:= x^{2r-1}p^{r-2}\cdot \left( (1-r)[\frac{p^{3-r}}{x} + (r-2)\frac{p^{2-r}}{x} +\frac{p}{x}] - {(p-1)}^2[(x-1)(1-r)x^{1-r}+x^{1-r}(x^{r-1}-1)]\right),\\
\gamma &:=(x-1)\left( x^r\left( 
 \left(\frac{x}{p} \right)^{r-1}-1 \right) - \frac{x^r}{p}\left( 
 \left(\frac{x}{p} \right)^{r-2}-1 \right) \right)\cdot \left( (r-1)p^r -(r-2)p^{r-1} \right).
\end{align*}
Then, $\alpha = \beta + \gamma$.
\end{lemma}

\begin{proof}
    We have
\begin{align*}
\alpha &=p^rx^r\left( \frac{x^r}{p} -x^{r-1} - \frac{x}{p} + 1 \right) -p^{r-1}x^r \left( \frac{x^{r-1}}{p} -x^{r-2} -\frac{x}{p} +1 \right)\\
&= x^{2r}p^{r-1}-x^{2r-1}p^r -x^{r+1}p^{r-1} +x^rp^r -x^{2r-1}p^{r-2} +x^{2r-2}p^{r-1} +x^{r+1}p^{r-2}-x^rp^{r-1}\\
&= x^{2r} (p^{r-1}) + x^{2r-1} (-p^r-p^{r-2}) +x^{2r-2} (p^{r-1}) + x^{r+1} (p^{r-2}-p^{r-1}) +x^r(p^r -p^{r-1}),
\end{align*}
\begin{align*}
\frac{\beta}{x^{2r-1}p^{r-2}} &= (x^2-2x+1)[(1-r)x^{-1}p^{3-r}+(r-2)x^{-1}p^{2-r}+x^{-1}p]\\
&{\phantom{=\,}}-(p^2-2p)[x^{2-r}-rx^{2-r}+rx^{1-r}+1-2x^{1-r}]\\
&= (1-r)x^{3-r} + (r-2)xp^{2-r}+xp-2(1-r)p^{3-r}-2(r-2)p^{2-r}-2p\\
&{\phantom{=\,}}+(1-r)x^{-1}p^{3-r}+(r-2)x^{-1}p^{2-r}+x^{-1}p\\
&{\phantom{=\,}}-(p^2-1)[x^{2-r}-rx^{2-r}+rx^{1-r}+1-2x^{1-r}]\\
&{\phantom{=\,}}+2px^{2-r} -2prx^{2-r}+2prx^{1-r} +2p -4px^{1-r}\\
&= x\left( (1-r)p^{3-r}+(r-2)p^{2-r}+p \right) +1\cdot \left( (-2)(1-r)p^{3-r}-2(r-2)p^{2-r}-p^2-1 \right)\\
&{\phantom{=\,}}+x^{-1}\left( 
(1-r)p^{3-r}+(r-2)p^{2-r}+p \right) + x^{1-r}\left( -(p^2+1)(r-2)+2pr-4p \right)\\
&{\phantom{=\,}}+x^{2-r}\left( 
-(p^2-1)(1-r)+2p-2pr \right),
\end{align*}
\begin{align*}
\gamma &= (x-1)(x^{2r-1}p^{1-r}-x^r -x^{2r-2}p^{1-r}+x^rp^{-1})(rp^r-p^r-rp^{r-1}+2p^{r-1})\\
&=\left( x^{2r}p^{1-r} -x^{r+1}-x^{2r-1}p^{1-r} +x^{r+1}p^{-1}-x^{2r-1}p^{1-r}+x^r+x^{2r-2}p^{1-r}-x^rp^{-1} \right)\cdot\\
&{\phantom{=\,}}\cdot (rp^r-p^r-rp^{r-1}+2p^{r-1})\\
&=\left( x^{2r}(p^{1-r})+x^{2r-1}(-2p^{1-r})+x^{2r-2}(p^{1-r})+x^{r+1}(p^{-1}-1)+x^r(1-p^{-1})\right)\cdot\\
&{\phantom{=\,}}\cdot (rp^r-p^r-rp^{r-1}+2p^{r-1})\\
&=x^{r+1}(2rp^{r-1} -3p^{r-1} -rp^{r-2} +2p^{r-2} -rp^r +p^r)\\
&{\phantom{=\,}}+ x^r(rp^r-p^r-2rp^{r-1}+3p^{r-1}+rp^{r-2}-2p^{r-2})\\
&{\phantom{=\,}}+x^{2r-2}(pr-p-r+2) + x^{2r-1}(-2pr+2p+2r-4) +x^{2r}(pr-p-r+2).
\end{align*}
Hence,
\begin{align*}
\beta + \gamma &= x^{2r}\left( (1-r)p +(r-2) +p^{r-1} +pr-p-r+2\right)\\
&{\phantom{=\,}}+ x^{2r-1}\left( 
(-2+2r)p+4-2r-(p^2+1)p^{r-2}-2pr+2p+2r-4 \right)\\
&{\phantom{=\,}} + x^{2r-2} \left( (1-r)p +r-2+p^{r-1} +pr-p-r+2  \right)\\
&{\phantom{=\,}} + x^r \left( p^{r-2}(-p^2r+2p^2-r+2+2pr-4p)+p^{r-2}(r-2-2rp+3p+rp^2-p^2) \right)\\
&{\phantom{=\,}} + x^{r+1} \left( p^{r-2}(-p^2+rp^2-1+r+2p-2pr) +p^{r-2}(-r+2+2rp-3p-rp^2+p^2) \right)\\
&=x^{2r}(p^{r-1}) + x^{2r-1}\left(-p^{r-2}(p^2+1)\right) + x^{2r-2}(p^{r-1}) + x^rp^{r-2}(p^2-p) + x^{r+1}p^{r-2}(1-p)\\
&=\alpha .
\end{align*}
\end{proof}

\begin{theor}\label{thm:sigma}
The numbers $\sigma_{n}(\ell)$, $n$ even, $\ell=1,\ldots,r$ satisfy (for $m>2$):
\begin{eqnarray*}
\sigma_2(\ell) &=& \delta_{\ell,r}, \quad\quad
\sigma_{4}(\ell)=\begin{cases}
0 & \text{ if}\quad \ell=1, \\
(\ell-1)(p^r-p^{r-1}) & \text{ if}\quad 1<\ell<r, \\
(r-1)p^r-(r-2)p^{r-1} & \text{ if}\quad \ell=r.
\end{cases}
\\
\sigma_{2m}(\ell)&=&\begin{cases}
0 & \text{ if}\quad \ell=1, \\
p^{{\left(2 \, m - 3\right)} r -\ell {\left(m - 2\right)} + 1-m} (p^{m-1}-1) [\ell-1]_{p^{m-2}}
& \text{ if}\quad 1<\ell<r, \\
p^{{\left(m - 1\right)} r} [r-1]_{p^{m-2}} - p^{{\left(m - 1\right)} r - 1} [r-2]_{p^{m-2}}
& \text{ if}\quad \ell=r.
\end{cases}
\end{eqnarray*}
\end{theor}
\begin{proof}
To begin with, we prove the base case. For $\ell\in \{1,\ldots, r\}$, we have 
$$|\sigma_2(\ell)| = |\{(c_1,c_2)\in I^2\ |\ [c_1,c_2] = \begin{pmatrix}
    -1+p^\ell & 0\\
    0 & \frac{1}{-1+p^\ell}
\end{pmatrix}\}|.$$
As $[c_1,c_2] = \begin{pmatrix}
    c_1c_2 -1 & -c_1\\
    c_2 & -1
\end{pmatrix}$, we deduce that $c_1=c_2=0$ such that there exists a solution if and only if $\ell =r$ in which case the solution is unique modulo $p^r$. Similarly, for $\ell\in \{1,\ldots, r\}$, we have 
$$|\sigma_4(\ell)| = |\{(c_1,c_2, c_3, c_4)\in I^4\ |\ [c_1,c_2,c_3,c_4] = \begin{pmatrix}
    1+p^\ell & 0\\
    0 & \frac{1}{1+p^\ell}
\end{pmatrix}\}|.$$
As $[c_1,c_2,c_3,c_3] = \begin{pmatrix}
    c_1c_2c_3c_4 -c_1c_4 -c_3c_4 -c_1c_2 +1 & -c_1c_2c_3 +c_1+c_3\\
    c_2c_3c_4 -c_4-c_2 & -c_2c_3 + 1
\end{pmatrix}$, we obtain
$$\begin{pmatrix}
    c_1c_2c_3c_4 -c_1c_4 -c_3c_4 -c_1c_2 +1 & c_1(c_2c_3-1)\\
    c_2(c_3c_4 -1) & -c_2c_3 + 1
\end{pmatrix} = \begin{pmatrix}
    1+p^\ell & c_3\\
    c_4 & \frac{1}{1+p^\ell}
\end{pmatrix}.$$
As $c_2c_3-1, c_3c_4-1\in R^\times$, we deduce that

\begin{equation}\label{eq:BaseCase}
c_2=\frac{c_4}{c_3c_4-1}\quad\textup{and}\quad c_1 = \frac{c_3}{c_2c_3-1}=c_3^2c_4-c_3.
\end{equation}
Hence, the two remaining equations 
$c_1c_2c_3c_4 -c_1c_4 -c_3c_4 -c_1c_2 +1=1+p^\ell$ and $-c_2c_3+1=\frac{1}{1+p^\ell}$ become both equivalent to the equation $p^\ell =-c_3c_4$. Note that any choice of $c_3$ and $c_4$ uniquely determines $c_1$ and $c_2$ modulo $p^r$ by (\ref{eq:BaseCase}). For $2\leq \ell \leq r $ the claim follows now from Lemma \ref{lem:mu_uv} and for $\ell =1$ the claim follows from the fact that $c_3, c_4\in I$.\par
\noindent Next, assume that $1<\ell<r$ and consider the right hand side of the recursion for $n=2m$. The two sums and the last summand are:
\begin{eqnarray*}
&& \sum_{j=1}^{\ell-1} n_j \sigma_{n}(j)\mu(j)\\
&=& \sum_{j=2}^{\ell-1} n_j p^{{\left(2 \, m - 3\right)} r -j {\left(m - 2\right)} + 1-m} (p^{m-1}-1) [j-1]_{p^{m-2}}(p^r-p^{r-1})(j-1) \\
&=& \sum_{j=2}^{\ell-1} p^{{\left(2 \, m - 3\right)} r -j {\left(m - 2\right)} + 1-m - j} (p^{m-1}-1) [j-1]_{p^{m-2}}(p^r-p^{r-1})^2(j-1) \\
&=& \sum_{j=2}^{\ell-1} p^{{\left(2 \, m - 3\right)} r -j {\left(m - 2\right)} + 1-m - j} (p^{m-1}-1) \frac{p^{(j-1)(m-2)}-1}{p^{m-2}-1}(p^r-p^{r-1})^2(j-1) \\
&=& \frac{p^{m-1}-1}{p^{m-2}-1} (p^r-p^{r-1})^2 p^{(2m-3)r+1-m} \sum_{j=2}^{\ell-1} p^{-j(m-1)} (p^{(j-1)(m-2)}-1) (j-1) \\
&=& \frac{p^{m-1}-1}{p^{m-2}-1} (p^r-p^{r-1})^2 p^{(2m-3)r+1-m} \biggl(
\frac{-(\ell-1)p^{4-\ell-m}+(\ell-2)p^{3-\ell-m}+p^{2-m}}{(p-1)^2} \\
&& - \frac{-(\ell-1)p^{(m-1)(2-\ell)}+(\ell-2)p^{(m-1)(1-\ell)}+1}{(p^{m-1}-1)^2}\biggr).
\end{eqnarray*}
\begin{eqnarray*}
&& \sum_{j=\ell+1}^{r} n_j \sigma_{n}(j)\mu(\ell)\\
&=& \sum_{j=\ell+1}^{r-1} n_j p^{{\left(2 \, m - 3\right)} r -j {\left(m - 2\right)} + 1-m} (p^{m-1}-1) [j-1]_{p^{m-2}}(p^r-p^{r-1})(\ell-1)\\
&&+ n_r (p^{(m-1)r}[r-1]_{p^{m-2}} - p^{(m-1)r-1}[r-2]_{p^{m-2}})(p^r-p^{r-1})(\ell-1) \\
&=& \sum_{j=\ell+1}^{r-1} p^{{\left(2 \, m - 3\right)} r -j {\left(m - 2\right)} + 1-m-j} (p^{m-1}-1) [j-1]_{p^{m-2}}(p^r-p^{r-1})^2(\ell-1)\\
&&+ (p^{(m-1)r}[r-1]_{p^{m-2}} - p^{(m-1)r-1}[r-2]_{p^{m-2}})(p^r-p^{r-1})(\ell-1) \\
&=& \sum_{j=\ell+1}^{r-1} p^{{\left(2 \, m - 3\right)} r -j {\left(m - 2\right)} + 1-m-j} (p^{m-1}-1) \frac{p^{(j-1)(m-2)}-1}{p^{m-2}-1}(p^r-p^{r-1})^2(\ell-1)\\
&&+ (p^{(m-1)r}\frac{p^{(r-1)(m-2)}-1}{p^{m-2}-1} - p^{(m-1)r-1}\frac{p^{(r-2)(m-2)}-1}{p^{m-2}-1})(p^r-p^{r-1})(\ell-1) \\
&=& \frac{p^{m-1}-1}{p^{m-2}-1} (p^r-p^{r-1})^2(\ell-1) p^{(2m-3)r+1-m} \sum_{j=\ell+1}^{r-1} p^{-j(m-1)} (p^{(j-1)(m-2)}-1)\\
&&+ (p^{(m-1)r}\frac{p^{(r-1)(m-2)}-1}{p^{m-2}-1} - p^{(m-1)r-1}\frac{p^{(r-2)(m-2)}-1}{p^{m-2}-1})(p^r-p^{r-1})(\ell-1) \\
&=& \frac{p^{m-1}-1}{p^{m-2}-1} (p^r-p^{r-1})^2(\ell-1) p^{(2m-3)r+1-m}
\left( p^{2-m}\frac{p^{-r}-p^{-\ell-1}}{p^{-1}-1}
-\frac{p^{(1-m)r}-p^{(1-m)(\ell+1)}}{p^{1-m}-1}\right) \\
&&+ (p^{(m-1)r}\frac{p^{(r-1)(m-2)}-1}{p^{m-2}-1} - p^{(m-1)r-1}\frac{p^{(r-2)(m-2)}-1}{p^{m-2}-1})(p^r-p^{r-1})(\ell-1).
\end{eqnarray*}
\begin{eqnarray*}
&& \sigma_{n}(\ell) \left(\mu(\ell)(p-2)p^{r-\ell-1} + \sum_{j=\ell+1}^{r} n_j \mu(j)\right) \\
&=& p^{{\left(2 \, m - 3\right)} r -\ell {\left(m - 2\right)} + 1-m} (p^{m-1}-1) [\ell-1]_{p^{m-2}} \left(\mu(\ell)(p-2)p^{r-\ell-1} + \sum_{j=\ell+1}^{r} n_j \mu(j)\right) \\
&=& p^{{\left(2 \, m - 3\right)} r -\ell {\left(m - 2\right)} + 1-m} (p^{m-1}-1) \frac{p^{(\ell-1)(m-2)}-1}{p^{m-2}-1} \left(\mu(\ell)(p-2)p^{r-\ell-1} + \sum_{j=\ell+1}^{r} n_j \mu(j)\right) \\
&=& p^{{\left(2 \, m - 3\right)} r -\ell {\left(m - 2\right)} + 1-m} (p^{m-1}-1) \frac{p^{(\ell-1)(m-2)}-1}{p^{m-2}-1} \left((\ell-1)p^{2r-\ell}-(2\ell-3)p^{2r-\ell-1}+(\ell-1)p^{2r-\ell-2}\right).
\end{eqnarray*}
With $x:=p^m$, $y:=p^r$, $z:=p^\ell$, $u:=p^{mr}$, $v:=p^{m\ell}$ the equations become
\begin{eqnarray*}
&& \sum_{j=1}^{\ell-1} n_j \sigma_{n}(j)\mu(j) + \sum_{j=\ell+1}^{r} n_j \sigma_{n}(j)\mu(\ell)\\
&&+ \sigma_{n}(\ell) \left(\mu(\ell)(p-2)p^{r-\ell-1} + \sum_{j=\ell+1}^{r} n_j \mu(j)\right) \\
&=& \frac{p^{m-1}-1}{p^{m-2}-1} (p^r-p^{r-1})^2 p^{(2m-3)r+1-m} \biggl(
\frac{-(\ell-1)p^{4-\ell-m}+(\ell-2)p^{3-\ell-m}+p^{2-m}}{(p-1)^2} \\
&& - \frac{-(\ell-1)p^{(m-1)(2-\ell)}+(\ell-2)p^{(m-1)(1-\ell)}+1}{(p^{m-1}-1)^2}\biggr)\\
&& +\frac{p^{m-1}-1}{p^{m-2}-1} (p^r-p^{r-1})^2(\ell-1) p^{(2m-3)r+1-m}
\left( p^{2-m}\frac{p^{-r}-p^{-\ell-1}}{p^{-1}-1}
-\frac{p^{(1-m)r}-p^{(1-m)(\ell+1)}}{p^{1-m}-1}\right) \\
&&+ (p^{(m-1)r}\frac{p^{(r-1)(m-2)}-1}{p^{m-2}-1} - p^{(m-1)r-1}\frac{p^{(r-2)(m-2)}-1}{p^{m-2}-1})(p^r-p^{r-1})(\ell-1) \\
&&+ p^{{\left(2 \, m - 3\right)} r -\ell {\left(m - 2\right)} + 1-m} (p^{m-1}-1) \frac{p^{(\ell-1)(m-2)}-1}{p^{m-2}-1} \left((\ell-1)p^{2r-\ell}-(2\ell-3)p^{2r-\ell-1}+(\ell-1)p^{2r-\ell-2}\right) \\
&=& {\left(\ell - 1\right)} {\left(y - \frac{y}{p}\right)} {\left(\frac{u {\left(\frac{p^{2} u}{x y^{2}} - 1\right)}}{y {\left(\frac{x}{p^{2}} - 1\right)}} - \frac{u {\left(\frac{p^{4} u}{x^{2} y^{2}} - 1\right)}}{p y {\left(\frac{x}{p^{2}} - 1\right)}}\right)} \\
&& +\frac{{\left(\ell - 1\right)} p u^{2} {\left(y - \frac{y}{p}\right)}^{2} {\left(\frac{p^{2} {\left(\frac{1}{y} - \frac{1}{p z}\right)}}{x {\left(\frac{1}{p} - 1\right)}} - \frac{\frac{y}{u} - \frac{p z}{v x}}{\frac{p}{x} - 1}\right)} {\left(\frac{x}{p} - 1\right)}}{x y^{3} {\left(\frac{x}{p^{2}} - 1\right)}} \\
&& + \frac{{\left(\frac{{\left(\ell - 1\right)} y^{2}}{z} - \frac{{\left(2 \, \ell - 3\right)} y^{2}}{p z} + \frac{{\left(\ell - 1\right)} y^{2}}{p^{2} z}\right)} p u^{2} z^{2} {\left(\frac{x}{p} - 1\right)} {\left(\frac{p^{2} v}{x z^{2}} - 1\right)}}{v x y^{3} {\left(\frac{x}{p^{2}} - 1\right)}} \\
&& - \frac{p u^{2} {\left(y - \frac{y}{p}\right)}^{2} {\left(\frac{x}{p} - 1\right)} {\left(\frac{\frac{{\left(\ell - 1\right)} p^{4}}{x z} - \frac{{\left(\ell - 2\right)} p^{3}}{x z} - \frac{p^{2}}{x}}{{\left(p - 1\right)}^{2}} + \frac{\frac{{\left(\ell - 2\right)} x z}{p v} - \frac{{\left(\ell - 1\right)} x^{2} z}{p^{2} v} + 1}{{\left(\frac{x}{p} - 1\right)}^{2}}\right)}}{x y^{3} {\left(\frac{x}{p^{2}} - 1\right)}} \\
&\stackrel{(*)}{=}& \frac{u^{2} {\left(x - 1\right)} z {\left(\frac{p v}{x z} - 1\right)}}{v x y {\left(\frac{x}{p} - 1\right)}} \\
&=& p^{{\left(2 \, m - 1\right)} r -\ell {\left(m - 1\right)} -m} (p^{m}-1) \frac{p^{(m-1)(\ell-1)}-1}{p^{m-1}-1} \\
&=& p^{{\left(2 \, m - 1\right)} r -\ell {\left(m - 1\right)} -m} (p^{m}-1) [\ell-1]_{p^{m-1}} = \sigma_{n+2}(\ell)
\end{eqnarray*}
where $(*)$ may be checked with a computer (e.g.\ with \textsc{Sage}) since these are rational functions with constant exponents. For the sake of completeness, we have included a proof in Appendix \ref{sec:appendix_proof_of_star}.

\indent The case $\ell=1$ follows from the fact that $\sigma_4(1)=0$ and from Theorem \ref{thm:recusion_for_sigma_n}, as $\mu(1)=0$.

\noindent Next, we consider the case $\ell =r$. Accoring to Theorem \ref{thm:recusion_for_sigma_n} we have to prove that
$$\sigma_{2(m+1)}(r):=p^{mr}[r-1]_{p^{m-1}} - p^{mr-1}[r-2]_{p^{m-1}} = \sum\limits_{j=1}^{r}n_j\sigma_{2m}(j)\mu(j).$$
We compute the right-hand side of the previous equation:
\begin{align*}
&\phantom{=\,}\sum\limits_{j=1}^{r}n_j\sigma_{2m}(j)\mu(j)= 0+ \sum\limits_{j=2}^{r-1}n_j\sigma_{2m}(j)\mu(j)\\
    &+ \left( p^{(m-1)r}[r-1]_{p^{m-2}} -p^{(m-1)r-1}[r-2]_{p^{m-2}}\right)\cdot \left( (r-1)p^r-(r-2)p^{r-1} \right)\\
    &=\sum\limits_{j=2}^{r-1}(p^{r-j}-p^{r-j-1})\cdot p^{(2m-3)r+1-m}(p^{m-1}-1)p^{-j(m-2)}[j-1]_{p^{m-2}}(j-1)(p^r-p^{r-1})\\
    &+ \left( p^{(m-1)r}[r-1]_{p^{m-2}} -p^{(m-1)r-1}[r-2]_{p^{m-2}}\right)\cdot \left( (r-1)p^r-(r-2)p^{r-1} \right)\\
    &={(p^r-p^{r-1})}^2\cdot p^{(2m-3)r+1-m}(p^{m-1}-1)\sum\limits_{j=2}^{r-1}p^{-j(m-1)}[j-1]_{p^{m-2}}(j-1)\\
    &+ \left( p^{(m-1)r}[r-1]_{p^{m-2}} -p^{(m-1)r-1}[r-2]_{p^{m-2}}\right)\cdot \left( (r-1)p^r-(r-2)p^{r-1} \right)\\
    &\stackrel{\textup{Lemma} \ref{lem:formula_for_sum_not_only_involving_geometric_things_but_simultaneously_also_involving_q_integers}}{=} \frac{{(p-1)}^2(p^{m-1}-1)p^{2mr-r-m-1}}{p^{m-2}-1}\left( \frac{(1-r)p^{4-r-m}+(r-2)p^{3-r-m}+p^{2-m}}{{(p-1)}^2}\right.\\
&\left.- \frac{(1-r)p^{(m-1)(2-r)} + (r-2)p^{(m-1)(1-r)}+1}{{(p^{m-1}-1)}^2} \right)\\
    &+ \left( p^{(m-1)r}[r-1]_{p^{m-2}} -p^{(m-1)r-1}[r-2]_{p^{m-2}}\right)\cdot \left( (r-1)p^r-(r-2)p^{r-1} \right)\\
&\stackrel{\textup{Lemma} \ref{lem:polynomial_long_division}}{=} \frac{{(p-1)}^2(p^{m-1}-1)p^{2mr-r-m-1}}{p^{m-2}-1}\left( \frac{(1-r)p^{4-r-m}+(r-2)p^{3-r-m}+p^{2-m}}{{(p-1)}^2}\right.\\
&\left.- \frac{(1-r)p^{(m-1)(1-r)} + p^{(m-1)(1-r)}\cdot [r-1]_{p^{m-1}}}{{(p^{m-1}-1)}} \right)\\
    &+ \left( p^{(m-1)r}[r-1]_{p^{m-2}} -p^{(m-1)r-1}[r-2]_{p^{m-2}}\right)\cdot \left( (r-1)p^r-(r-2)p^{r-1} \right)\\
    &=\frac{p^{2mr-r-m-1}}{p^{m-2}-1}\left( (p^{m-1}-1)[(1-r)p^{4-r-m}+(r-2)p^{3-r-m}+p^{2-m}]\right.\\
    &\left. - {(p-1)}^2 [(1-r)p^{(m-1)(1-r)} + p^{(m-1)(1-r)}[r-1]_{p^{m-1}})]\right)\\
    &+ \left( p^{(m-1)r}[r-1]_{p^{m-2}}-p^{(m-1)r-1}[r-2]_{p^{m-2}}\right)\cdot \left( (r-1)p^r-(r-2)p^{r-1} \right).
\end{align*}
Next, we set $x:=p^{m-1}$. Then, our equation becomes
\begin{align*}
    {(px)}^r[r-1]_x - \frac{{(px)}^r}{p}[r-2]_x &= \frac{x^{2r}p^{r-2}x^{-1}}{\frac{x}{p}-1}\cdot \left((x-1)[(1-r)x^{-1}p^{3-r}+(r-2)x^{-1}p^{2-r}+x^{-1}p]\right.\\
&\left. -{(p-1)}^2[(1-r)x^{1-r}+x^{1-r}[r-1]_x]\right)\\
    &+\left(x^r[r-1]_{\frac{x}{p}} - \frac{x^r}{p}[r-2]_{\frac{x}{p}}\right)\cdot \left((r-1)p^r -(r-2)p^{r-1}\right).
\end{align*}
Now, we replace the $q$-integers by fractions. Then, the last equation is equivalent to:
\begin{align*}
p^rx^r\frac{x^{r-1}-1}{x-1} -p^{r-1}x^r\frac{x^{r-2}-1}{x-1} &= \frac{x^{2r-1}p^{r-2}}{\frac{x}{p}-1}\cdot \left((x-1)[(1-r)x^{-1}p^{3-r}+(r-2)x^{-1}p^{2-r}+x^{-1}p]\right.\\
&\left. -{(p-1)}^2[(1-r)x^{1-r}+x^{1-r}\frac{x^{r-1}-1}{x-1}]\right)\\
    &+\left(x^r\frac{{(\frac{x}{p})}^{r-1}-1}{\frac{x}{p}-1} - \frac{x^r}{p}\cdot \frac{{(\frac{x}{p})}^{r-2}-1}{\frac{x}{p}-1}\right)\cdot \left( (r-1)p^r -(r-2)p^{r-1} \right).
\end{align*}
Next, we multiply both sides of the last equation by $(x-1)\cdot (\frac{x}{p}-1)$ and obtain:
\begin{align*}
&{\phantom{=}} p^rx^r(x^{r-1}-1)(\frac{x}{p}-1)-p^{r-1}x^r(x^{r-2}-1)(\frac{x}{p}-1)= x^{2r-1}p^{r-2}\cdot\\
&\cdot \left( (1-r)[x^{-1}p^{3-r} + (r-2)x^{-1}p^{2-r} +x^{-1}p] - {(p-1)}^2[(x-1)(1-r)x^{1-r}+x^{1-r}(x^{r-1}-1)]\right)\\
&+ (x-1)\left( x^r\left( 
 \left(\frac{x}{p} \right)^{r-1}-1 \right) - \frac{x^r}{p}\left( 
 \left(\frac{x}{p} \right)^{r-2}-1 \right) \right)\cdot \left( (r-1)p^r -(r-2)p^{r-1} \right).
\end{align*}
The correctness of this last equation follows from Lemma \ref{lem:alpha_beta_gamma}.
\end{proof}

\begin{theor}\label{thm:nfriezes_even}
Let $R=\mathbb{Z}/p^r\mathbb{Z}$ for a prime $p$ and $r\in\mathbb{N}$, $I=pR$ the maximal ideal, $n\in\mathbb{N}_{>1}$ with $n$ even, and $\varepsilon\in\{\pm 1\}$.
Let $\omega_n$ be the number of $\varepsilon$-quiddity cycles over $R/I$ of length $n$.
If $p>2$ then the number of $\varepsilon$-quiddity cycles over $R$ is
$$
\begin{cases}
(\omega_n-1) \cdot p^{(r-1)(n-3)} + \sigma_n(r) & \quad \text{if}\quad \varepsilon=(-1)^{n/2}, \\
\omega_n \cdot p^{(r-1)(n-3)} & \quad \text{if}\quad \varepsilon=-(-1)^{n/2}.
\end{cases}
$$
If $p=2$ then the number of $\varepsilon$-quiddity cycles over $R$ is
$$
\begin{cases}
(\omega_n-1) \cdot p^{(r-1)(n-3)} + \sigma_n(r) & \quad \text{if}\quad \varepsilon=(-1)^{n/2}, \\
(\omega_n-1) \cdot p^{(r-1)(n-3)} & \quad \text{if}\quad \varepsilon=-(-1)^{n/2}.
\end{cases}
$$
\end{theor}
\begin{proof}
We proceed as in the proof of Theorem \ref{thm:nfriezesI}:
If $\kappa : R \longrightarrow R/I$ is the canonical projection and $(c_1,\ldots,c_n)$ is an $\varepsilon$-quiddity cycle, then
$(\kappa(c_1),\ldots,\kappa(c_n))$ is a $\kappa(\varepsilon)$-quiddity cycle over $R/I$.
For each fixed $\varepsilon$-quiddity cycle $\underline{c}=(\overline{c_1},\ldots,\overline{c_n})\in (R/I)^n$, we count the number of $\varepsilon$-quiddity cycles
which project to this cycle under $\kappa$.

Since $I$ is maximal, $R/I$ is a field. Assume first that $\underline{c}$ contains no unit, i.e.\ $\underline{c}=(0,\ldots,0)$. Then
the number of cycles which map to $\underline{c}$ under $\kappa$ is $\sigma_{n}(r)$ if $\varepsilon=(-1)^{n/2}$. If $\varepsilon\ne (-1)^{n/2}$, then there is no such solution.

Otherwise $\underline{c}\ne(0,\ldots,0)$.
There are $\omega_n-1$ such cases if $\varepsilon=(-1)^{n/2}$;
there are $\omega_n$  such cases if $\varepsilon=-(-1)^{n/2}$ and $p>2$; and $\omega_n-1$ such cases if $\varepsilon=-(-1)^{n/2}$ and $p=2$ (note that $\omega_n$ also counts $(0,\ldots,0)$ in this case):
we may assume that there is an entry in $\underline{c}$ which is a unit, after rotating the cycle, without loss of generality $\overline{c_2}\in (R/I)^\times$.
Now the same argument as in the proof of Theorem \ref{thm:nfriezesI} produces $|I|^{n-3}=p^{(r-1)(n-3)}$ different $\varepsilon$-quiddity cycles of length $n$ in each case.
\end{proof}

\section{Cycles in residue class rings}\label{sec:recursion}
\newcommand{\udi}[1]{(#1)}

In this last section we present a recursion for the number of solutions
\begin{equation}
\pi_{x,n} := |\{ (c_1,\ldots,c_n)\in R^n \mid  \ql c_1,\ldots,c_n\qr=A \text{ and } x=c_2 \}|.
\end{equation}
for $R:=\ZZ/p^r\ZZ$, $A\in R^{2\times 2}$, and $n\ge 3$.
A sum over all $x\in R$ then gives a recursion for the number of solutions $\ql c_1,\ldots,c_n\qr=A$.

As an application one could recover the formulae from the previous sections; however, all in all this would result in a longer proof than before. Since this section is less important, we omit some of the technical proofs and leave them to the reader.
\medskip

To count the different cases we need the following numbers.
\begin{defin}
For $x,u\in R$, let
\begin{eqnarray}
\xi_{x,u} &:=& \begin{cases} p^{\nu_p(u)} & \nu_p(u)\le \nu_p(x+1) \\
                           0 & \nu_p(u) > \nu_p(x+1) \end{cases}, \\
\zeta_{x,u} &:=& \begin{cases} 0 & x+u\in R^\times \text{ or } x\notin R^\times \\
                       rp^r-(r-1)p^{r-1} & x+u = 0 \\
                       (p^r-p^{r-1})\nu_p(x+u) & \nu_p(x+u)>0 \text{ and } x+u\ne 0 \end{cases}
\end{eqnarray}
where $\nu_p(a)$ is the largest $k$ such that $a\equiv 0 \md{p^k}$.
\end{defin}

We obtain:
\begin{propo}\label{prop:recursion}
For $u\in R$ and $n\ge 5$, we have the recursion
\begin{equation}
\pi_{u,n} = \sum_{x\in R^\times} \pi_{x,n-1} \xi_{x,u} + \pi_{x,n-2} \zeta_{x,u}.
\end{equation}
\end{propo}
\begin{proof}
Let $(c,u,v,b,d,\ldots)$ be a sequence of length $n$.
There are two cases:
\begin{enumerate}
\item[(a)] $uv-1$ is a unit. Then we can reduce the sequence to a sequence of length $n-1$ via
\[ [c,u,v,b,\ldots] =
  [c+(1-v)/(uv-1), uv-1, b+(1-u)/(uv-1),\ldots] \]
\item[(b)] $uv-1$ is not a unit. Then $v$ is a unit and we can reduce the sequence to a sequence of length $n-2$ via
\[ [c,u,v,b,d,\ldots] =
  [c-(vb-2)/x, x, d-(uv-2)/x,\ldots] \]
for $x = ((vb-1)(uv-1)-1)/v$.
\end{enumerate}
Thus every sequence comes from a shorter sequence after a step of type (a) or (b):

Assume we have a sequence of length $n-1$ with second entry $x\in R^\times$. If this sequence was obtained via a reduction as in (a), then $x=uv-1$ for some $u,v\in R$.
If the longer sequence is counted by $\pi_{u,n}$, then $\nu_p(u) > \nu_p(x+1)$ is excluded, and if $\nu_p(u)\le \nu_p(x+1)$ then there are $p^{\nu_p(u)}$ possible $u\in R$ that satisfy $x=uv-1$. This explains the summand $\pi_{x,n-1} \xi_{x,u}$ in the recursion.

Now assume that we have a sequence of length $n-2$ with second entry $x\in R^\times$. If this sequence was obtained via a reduction as in (b), then $x=((vb-1)(uv-1)-1)/v$ for some $u,v,b\in R$ such that $uv-1$ is not a unit.
In particular, $x+u=b(uv-1)$ is not a unit. If $x+u=0$, then there are $rp^r-(r-1)p^{r-1}$ triple $u,v,b$ that satisfy the relations; if $\nu_p(x+u)>0$ and $x+u\ne 0$, then we have $(p^r-p^{r-1})\nu_p(x+u)$ solutions. This is why the summand $\pi_{x,n-2} \zeta_{x,u}$ appears in the recursion.

Note that the cases (a) and (b) always produce a unit at the second position in the shorter sequence, thus we may sum over $x\in R^\times$ in the recursion.
\end{proof}

\begin{defin}
We use the sets
\begin{eqnarray}
u_{d,r} &:=& \{d\}, \\
u_{d,i} &:=& \{u \mid u\equiv d \md{p^i}, u\not\equiv d \md{p^{i+1}} \}, \\
e &:=& \{u \mid u \not\equiv -1,0,1 \md{p}\},
\end{eqnarray}
for $d=-1,0,1$ and $i=1,\ldots,r-1$.
Since $R = e \cup \bigcup_{d,i} u_{d,i}$ is a disjoint union, we obtain an equivalence relation on $R$.
For $u\in R$, denote $\udi{u}$ the class of $u$ with respect to this relation, and
write
\begin{eqnarray}
\pi_{\udi{x},n} &:=& \sum_{u\in \udi{x}} \pi_{u,n}, \\
\xi_{\udi{x},\udi{u}} &:=& \sum_{v\in \udi{x}, w\in \udi{u}} \xi_{v,w}, \\
\zeta_{\udi{x},\udi{u}} &:=& \sum_{v\in \udi{x}, w\in \udi{u}} \zeta_{v,w}.
\end{eqnarray}
\end{defin}

The following proposition and corollary explain why these definitions are useful. We omit the proofs since they do not give any relevant new insights.

\begin{propo}
For $x,u\in R$ and $y_1,y_2\in\udi{x}$,
$$\sum_{z\in \udi{u}} \xi_{y_1,z} = \sum_{z\in \udi{u}} \xi_{y_2,z}, \quad\quad
\sum_{z\in \udi{u}} \zeta_{y_1,z} = \sum_{z\in \udi{u}} \zeta_{y_2,z}.$$
\end{propo}

\begin{corol}
For $x,u\in R$, we have $\pi_{\udi{x},n} = |\udi{x}|\cdot \pi_{x,n}$ and
\begin{equation}
\pi_{\udi{u},n} = \sum_{\udi{x},\:\: x\in R^\times} \frac{1}{|\udi{x}|} (\pi_{\udi{x},n-1} \xi_{\udi{x},\udi{u}} + \pi_{\udi{x},n-2} \zeta_{\udi{x},\udi{u}}).
\end{equation}
\end{corol}

We can now compute the required values:

\begin{propo}
Let $m:=|R^x|=p^r-p^{r-1}$. We have
\[ |e| = p^r-3p^{r-1}, \quad
n_i := |u_{d,i}|=\begin{cases} p^{r-i}-p^{r-i-1} & i<r \\ 1 & r=i \end{cases}, \]
and for $0<i,j<r$, $i\ne j$, $0\le k\le r$, $k\le \ell$, $x\in R$ we obtain:
\begin{eqnarray}
\xi_{(x),u_{\pm 1,k}} &=& |(x)| \cdot |u_{\pm 1,k}|, \\
\xi_{(x),e} &=& |(x)| \cdot |e|, \\
\xi_{u_{-1,\ell},u_{0,k}} &=& p^k\cdot |u_{-1,\ell}|\cdot |u_{0,k}|,
\end{eqnarray}
and $\xi_{(*),(*)}=0$ in all the remaining cases, and
\begin{eqnarray}
\zeta_{u_{1,i},u_{-1,j}} &=& \zeta_{u_{-1,j},u_{1,i}} = \min(i,j)\cdot m \cdot |u_{1,i}|\cdot |u_{-1,j}|, \\
                         &=& \min(i,j) ((p^{3r-i-j}-p^{3r-i-j-3})-3(p^{3r-i-j-1}-p^{3r-i-j-2})), \\
\zeta_{u_{1,i},u_{-1,i}} &=& \zeta_{u_{-1,i},u_{1,i}} = (i p^{2r-i}-(2i-1)p^{2r-i-1}+ip^{2r-i-2}) |u_{1,i}|, \\
                         &=& i (p^{3r-2i}-p^{3r-2i-3})-(3i-1)(p^{3r-2i-1}-p^{3r-2i-2}), \\
\zeta_{u_{1,r},u_{-1,r}} &=& \zeta_{u_{-1,r},u_{1,r}} = r p^r-(r-1) p^{r-1}, \\
\zeta_{e,e} &=& p^{2r-1} |e| = p^{2r-1} (p^r-3p^{r-1}),
\end{eqnarray}
and $\zeta_{(*),(*)}=0$ in all the remaining cases.
\end{propo}

In principle, for a fixed matrix $A$, one can use the above formulae to obtain closed formulae for the numbers of solutions.

As a last remark, we show that the sum over units satisfies a simpler recursion:

\begin{propo}\label{prop:tau}
Let $n\ge 5$ and
$$
\tau_n := \sum_{(u),\:\: u\in R^\times} \pi_{(u),n}.
$$
Then
$$
\tau_n = (p^r-p^{r-1}) \tau_{n-1} + p^{2r-1} \tau_{n-2}.
$$
\end{propo}
\begin{proof}
By Proposition \ref{prop:recursion}, we have
\begin{eqnarray*}
\tau_n &=& \sum_{u,x\in R^\times} \pi_{x,n-1} \xi_{x,u} + \pi_{x,n-2} \zeta_{x,u} \\
&=& \sum_{x\in R^\times} \pi_{x,n-1} \sum_{u\in R^\times} \xi_{x,u} + \pi_{x,n-2} \sum_{u\in R^\times}\zeta_{x,u}.
\end{eqnarray*}
For the first sum we get $\sum_{x\in R^\times} \pi_{x,n-1} \sum_{u\in R^\times} \xi_{x,u} = (p^r-p^{r-1})\sum_{x\in R^\times} \pi_{x,n-1}$ because $\xi_{x,u}=1$ if $u$ is a unit.
The second sum is
\begin{eqnarray*}
\sum_{x\in R^\times} \pi_{x,n-2} \sum_{u\in R^\times}\zeta_{x,u} &=& 
\sum_{x\in R^\times} \pi_{x,n-1} \left( \underbrace{(rp^r-(r-1)p^{r-1})}_{u=-x} + \sum_{k=1}^{r-1} \underbrace{(p^r-p^{r-1}) n_k k}_{\nu_p(x+u)=k} \right)\\
&=& \sum_{x\in R^\times} \pi_{x,n-1} p^{2r-1}.
\end{eqnarray*}
\end{proof}

\section{Appendix}\label{sec:appendix_proof_of_star}
\begin{propo}
    Let $p$ be a prime number, let $x\in\mathbb{R}\setminus \{0, p, \pm\sqrt{p}\}$, let $u,v,y,z\in\mathbb{R}\setminus \{0\}$, and let $\ell\in\mathbb{N}$. Then, the following equation holds:
\begin{eqnarray*}
\frac{u^2(x-1)z \left( \frac{pv}{xz} -1 \right)}{vxy\left( \frac{x}{p} -1 \right)} &=& (\ell -1) \left( y-\frac{y}{p} \right) \left( \frac{u \left( \frac{p^2u}{xy^2} -1 \right)}{y \left( \frac{x}{p^2} -1 \right)} - \frac{u\left( \frac{p^4u}{x^2y^2}-1\right) }{py\left( \frac{x}{p^2} -1 \right)} \right)\\
&+& \frac{(\ell -1)pu^2{\left( y - \frac{y}{p} \right)}^2\left( \frac{p^2\left(\frac{1}{y}-\frac{1}{pz}\right)}{x\left(\frac{1}{p} -1 \right)} - \frac{\frac{y}{u} - \frac{pz}{vx}}{\frac{p}{x}-1} \right) \left( \frac{x}{p} -1 \right)}{xy^3\left( \frac{x}{p^2}-1 \right)}\\
&+& \frac{\left(\frac{(\ell -1)y^2}{z} - \frac{(2\ell -3)y^2}{pz} + \frac{(\ell -1)y^2}{p^2z}\right) pu^2z^2 \left( \frac{x}{p}-1\right)\left( \frac{p^2v}{xz^2} -1 \right)}{vxy^3\left( \frac{x}{p^2}-1 \right)}\\
&-& \frac{pu^2{\left( y - \frac{y}{p} \right)}^2\left( \frac{x}{p}-1 \right) \left( \frac{ \frac{(\ell -1)p^4}{xz} - \frac{(\ell -2)p^3}{xz} - \frac{p^2}{x} }{{(p-1)}^2} + \frac{\frac{(\ell -2)xz}{pv} - \frac{(\ell -1)x^2z}{p^2v} +1 }{{\left(\frac{x}{p}-1\right)}^2} \right)}{xy^3\left( \frac{x}{p^2}-1 \right)}.
\end{eqnarray*}
\end{propo}
\begin{proof}
We multiply both sides of the equation by $vpxy^3\left( \frac{x}{p}-1 \right) \left( \frac{x}{p^2}-1 \right)$ and obtain the following equivalent assertion:
\begin{eqnarray*}
\phantom{0} &\phantom{=}& py^2\left( \frac{x}{p^2}-1 \right)u^2(x-1)z\left( \frac{pv}{xz} -1 \right)\\
&=& vpxy^2\left( \frac{x}{p}-1 \right) (\ell -1) \left( y- \frac{y}{p} \right) u \left( \frac{p^2u}{xy^2}-1 \right)\\
&-& vxy^2\left( \frac{x}{p}-1 \right)u\left( \frac{p^4u}{x^2y^2}-1 \right)(\ell -1)\left( y- \frac{y}{p} \right)\\
&+& vp {\left( \frac{x}{p} -1\right)}^2 (\ell -1)pu^2 {\left( y- \frac{y}{p} \right)}^2 \left( \frac{p^2\left(\frac{1}{y}-\frac{1}{pz}\right)}{x\left(\frac{1}{p} -1 \right)} - \frac{\frac{y}{u} - \frac{pz}{vx}}{\frac{p}{x}-1} \right)\\
&+& p \left( \frac{x}{p} -1\right)\left(\frac{(\ell -1)y^2}{z} - \frac{(2\ell -3)y^2}{pz} + \frac{(\ell -1)y^2}{p^2z}\right) pu^2z^2 \left( \frac{x}{p}-1\right)\left( \frac{p^2v}{xz^2} -1 \right)\\
&-& vp\left( \frac{x}{p} -1\right) pu^2{\left( y - \frac{y}{p} \right)}^2\left( \frac{x}{p}-1 \right) \left( \frac{ \frac{(\ell -1)p^4}{xz} - \frac{(\ell -2)p^3}{xz} - \frac{p^2}{x} }{{(p-1)}^2} + \frac{\frac{(\ell -2)xz}{pv} - \frac{(\ell -1)x^2z}{p^2v} +1 }{{\left(\frac{x}{p}-1\right)}^2} \right).
\end{eqnarray*}
This is equivalent to 
\begin{eqnarray*}
\phantom{0} &\phantom{=}& 
y^2\left(\frac{x}{p^2}-1\right)u^2(x-1)\left(\frac{p^2v}{x}-zp\right)\\
&=& v\left(\frac{x}{p}-1\right) (\ell -1)y(p-1)u(p^2u - xy^2) -v\left(\frac{x}{p}-1\right)u\left(\frac{p^4u}{x}-xy^2\right)(\ell -1)\left(y - \frac{y}{p}\right)\\
&+& v{\left(\frac{x}{p}-1\right)}^2(\ell -1)u^2y^2{(p-1)}^2 \left( \frac{p^2\left( \frac{p}{y} - \frac{1}{z} \right)}{x(1-p)} - \frac{\frac{xy}{u} - \frac{pz}{v}}{p-x} \right)\\
&+& (x-p)\left(\frac{(\ell -1)y^2}{z} - \frac{(2\ell -3)y^2}{pz} + \frac{(\ell -1)y^2}{p^2z}\right)(x-p)\left( \frac{p^2u^2v}{x} - u^2z^2 \right)\\
&-& v{\left(\frac{x}{p}-1\right)}^2u^2y^2{(p-1)}^2\left( \frac{ \frac{(\ell -1)p^4}{xz} - \frac{(\ell -2)p^3}{xz} - \frac{p^2}{x} }{{(p-1)}^2} + \frac{\frac{(\ell -2)xz}{pv} - \frac{(\ell -1)x^2z}{p^2v} +1 }{{\left(\frac{x}{p}-1\right)}^2} \right).
\end{eqnarray*}
This is equivalent to 
\begin{eqnarray*}
\phantom{0} &\phantom{=}& 
y^2\left(\frac{x}{p^2}-1\right)u^2(x-1)\left(\frac{p^2v}{x}-zp\right)\\
&=& v\left(\frac{x}{p}-1\right)(\ell -1)y(p-1)u(p^2u-xy^2) - v\left(\frac{x}{p}-1\right)u\left( \frac{p^4u}{x}-xy^2 \right) (\ell -1)\left(y-\frac{y}{p}\right)\\
&+& v{\left(\frac{x}{p}-1\right)}^2(\ell -1)u^2y^2(1-p)\frac{p^2}{x}\left( \frac{p}{y} - \frac{1}{z} \right)\\
&+& \frac{v}{p^2}(x-p)(\ell -1)u^2y^2{(1-p)}^2\left( \frac{xy}{u} - \frac{pz}{v} \right)\\
&+& (x-p)\left(\frac{(\ell -1)y^2}{z} - \frac{(2\ell -3)y^2}{pz} + \frac{(\ell -1)y^2}{p^2z}\right)(x-p)\left( \frac{p^2u^2v}{x} - u^2z^2 \right)\\
&+& v{\left( \frac{x}{p} -1 \right)}^2u^2y^2 \left( \frac{p^2}{x} + \frac{(\ell -2)p^3}{xz} - \frac{(\ell -1)p^4}{xz} \right) - vu^2y^2{(p-1)}^2\left( \frac{(\ell -2)xz}{pv} - \frac{(\ell -1)x^2z}{p^2v} + 1 \right).
\end{eqnarray*}
We simplify some parentheses and obtain:
\begin{eqnarray*}
\phantom{0} &\phantom{=}& 
y^2\left(\frac{x}{p^2}-1\right)u^2(x-1)\left(\frac{p^2v}{x}-zp\right)\\
&=& v\left(\frac{x}{p}-1\right)(\ell -1)y(p-1)u(p^2u-xy^2) - v\left(\frac{x}{p}-1\right)u\left( \frac{p^4u}{x}-xy^2 \right) (\ell -1)\left(y-\frac{y}{p}\right)\\
&+& v{\left(\frac{x}{p} -1\right)}^2(\ell -1)u^2y(1-p)\frac{p^3}{x} - v{\left(\frac{x}{p} -1\right)}^2(\ell -1)u^2y^2(1-p)\frac{p^2}{xz}\\
&+& \frac{v}{p^2}(x-p)(\ell -1){(1-p)}^2uxy^3 - \frac{z}{p} (x-p)(\ell -1)u^2y^2{(1-p)}^2\\
&+& (x-p)\left(\frac{(\ell -1)y^2}{z} - \frac{(2\ell -3)y^2}{pz} + \frac{(\ell -1)y^2}{p^2z}\right)(x-p)\left( \frac{p^2u^2v}{x} - u^2z^2 \right)\\
&+& v{\left( \frac{x}{p} -1 \right)}^2u^2y^2 \left( \frac{p^2}{x} + \frac{(\ell -2)p^3}{xz} - \frac{(\ell -1)p^4}{xz} \right) - u^2y^2{(p-1)}^2\left( \frac{(\ell -2)xz}{p} - \frac{(\ell -1)x^2z}{p^2} + v \right).
\end{eqnarray*}
We multiply both sides of the previous equation by $p^2xz$ and obtain:
\begin{eqnarray*}
\phantom{0} &\phantom{=}& u^2z(x-1)y^2(x-p^2)(p^2v-zpx)\\
&=& pxzv(x-p)(\ell -1)y(p-1)u(p^2u-xy^2) - zv(x-p)u(p^4u-x^2y^2)(\ell -1)y(p-1)\\
&+& vz{(x-p)}^2(\ell -1)u^2y(1-p)p^3 - v{(x-p)}^2(\ell -1)u^2y^2(1-p)p^2\\
&+& vz(x-p)(\ell -1){(1-p)}^2ux^2y^3 -pxz^2(x-p)(\ell -1)u^2y^2{(1-p)}^2\\
&+& (x-p)\left( p^2(\ell -1)y^2 - p(2\ell -3)y^2 + (\ell -1)y^2 \right)(x-p)(p^2u^2v -u^2z^2x)\\
&+& v{(x-p)}^2u^2y^2\left( p^2z +(\ell -2)p^3 - (\ell -1)p^4 \right)\\
&-& xzu^2y^2{(p-1)}^2 \left( (\ell -2)pxz - (\ell -1)x^2z +p^2v \right).
\end{eqnarray*}
After rearrangement, this equation becomes:
\begin{eqnarray*}
\phantom{0} &\phantom{=}& u^2z(x-1)y^2(x-p^2)(p^2v-zpx)\\
&=\phantom{(2\ell -3)}& p^2zv{(x-p)}^2u^2y^2 - p^2vxzu^2y^2{(p-1)}^2\\
&+\ (\ell -2)\cdot & [p^3u^2vy^2{(x-p)}^2 -px^2z^2u^2y^2{(p-1)}^2]\\
&-\ (2\ell -3)\cdot & [py^2{(x-p)}^2(p^2u^2v - u^2z^2x)]\\
&+\ (\ell -1)\cdot & [pxyzuv(x-p)(p-1)(p^2u-xy^2) - zuvy(p-1)(x-p)(p^4u -x^2y^2)]\\
&+\ (\ell -1)\cdot & [p^3u^2vyz(1-p){(x-p)}^2 - p^2u^2vy^2(1-p){(x-p)}^2]\\
&+\ (\ell -1)\cdot & [x^2y^3zuv{(1-p)}^2(x-p) - pxy^2z^2u^2{(1-p)}^2(x-p)]\\
&+\ (\ell -1)\cdot & [(p^2+1)y^2{(x-p)}^2p^2u^2v - (p^2+1)y^2{(x-p)}^2u^2z^2x]\\
&+\ (\ell -1)\cdot & [x^3z^2u^2y^2{(p-1)}^2 - p^4y^2u^2v{(x-p)}^2].
\end{eqnarray*}
This last equation follows from the Lemma below, as was to be verified.
\end{proof}

\begin{lemma}
    Let $p$ be a prime number, let $x,y,z,u,v\in\mathbb{R}$, and let $\ell \in \mathbb{N}$. Then, the following hold.
    \begin{itemize}
        \item[\rm (a)] We have
 \begin{eqnarray*}
\phantom{0} &\phantom{=}&  pu^2y^2\left(p^2v{(p-x)}^2 +xz^2\left( p^2(x-2) +2px -2x^2 +x \right) \right)\\ 
 &=& [pxyzuv(x-p)(p-1)(p^2u-xy^2) - zuvy(p-1)(x-p)(p^4u -x^2y^2)]\\
&+& [p^3u^2vyz(1-p){(x-p)}^2 - p^2u^2vy^2(1-p){(x-p)}^2]\\
&+& [x^2y^3zuv{(1-p)}^2(x-p) - pxy^2z^2u^2{(1-p)}^2(x-p)]\\
&+& [(p^2+1)y^2{(x-p)}^2p^2u^2v - (p^2+1)y^2{(x-p)}^2u^2z^2x]\\
&+& [x^3z^2u^2y^2{(p-1)}^2 - p^4y^2u^2v{(x-p)}^2].
 \end{eqnarray*}
\item[\rm (b)]
We have
 \begin{eqnarray*}
\phantom{0} &\phantom{=}&  pxy^2z^2u^2\left(p^2x-p^2 -x^2 +x\right)\\
&=& (\ell -1)\cdot \left[pu^2y^2\left(p^2v{(p-x)}^2 +xz^2\left( p^2(x-2) +2px -2x^2 +x \right) \right)\right]\\
&-& (2\ell -3)\cdot \left[ py^2{(x-p)}^2(p^2u^2v - u^2z^2x) \right]\\
&+& (\ell -2)\cdot \left[ p^3u^2vy^2{(x-p)}^2 -px^2z^2u^2y^2{(p-1)}^2 \right].
\end{eqnarray*}
\item[\rm (c)]
We have
 \begin{eqnarray*}
\phantom{0} &\phantom{=}&  u^2z(x-1)y^2(x-p^2)(p^2v-zpx)\\
&=& pxy^2z^2u^2\left(p^2x-p^2 -x^2 +x\right)\\
&+& p^2zv{(x-p)}^2u^2y^2 - p^2vxzu^2y^2{(p-1)}^2
\end{eqnarray*}
    \end{itemize}
\end{lemma}

\begin{proof}
    Direct computation.
\end{proof}

\bibliographystyle{amsalpha}
\bibliography{refs}

\providecommand{\bysame}{\leavevmode\hbox to3em{\hrulefill}\thinspace}
\providecommand{\MR}{\relax\ifhmode\unskip\space\fi MR }
\providecommand{\MRhref}[2]{%
  \href{http://www.ams.org/mathscinet-getitem?mr=#1}{#2}
}
\providecommand{\href}[2]{#2}
\begin{thebibliography}{Mab24}

\bibitem[CC73]{jChC73}
J.~H. Conway and H.~S.~M. Coxeter, \emph{Triangulated polygons and frieze
  patterns}, Math. Gaz. \textbf{57} (1973), 87--94 and 175--183.

\bibitem[CH19]{CH19}
M.~Cuntz and T.~Holm, \emph{Frieze patterns over integers and other subsets of
  the complex numbers}, J. Comb. Algebra \textbf{3} (2019), no.~2, 153--188.

\bibitem[CM24]{CM24}
Michael Cuntz and Flavien Mabilat, \emph{Comptage des quiddit{\'e}s sur les
  corps finis et sur quelques anneaux $\mathbb{Z}/n\mathbb{Z}$}, to appear in
  Ann. Fac. Sci. Toulouse Math. (6) (2024), 32 pp.

\bibitem[Cun19]{C19}
Michael Cuntz, \emph{A combinatorial model for tame frieze patterns},
  M\"{u}nster J. Math. \textbf{12} (2019), no.~1, 49--56.

\bibitem[Mab24]{M24}
Flavien Mabilat, \emph{Some counting formulas for $\lambda$-quiddities over the
  rings $\mathbb{Z}/2^{m}\mathbb{Z}$}, arXiv e-prints (2024), 10 p., available
  at \href{https://arxiv.org/abs/2402.09968}{\texttt{arXiv:2402.09968}}.

\bibitem[MG21]{MG21}
Sophie Morier-Genoud, \emph{Counting {C}oxeter's friezes over a finite field
  via moduli spaces}, Algebr. Comb. \textbf{4} (2021), no.~2, 225--240.

\bibitem[Ovs18]{O18}
Valentin Ovsienko, \emph{Partitions of unity in {$\SL(2,\mathbb{Z})$}, negative
  continued fractions, and dissections of polygons}, Res. Math. Sci. \textbf{5}
  (2018), no.~2, Paper No. 21, 25.

\end{thebibliography}

\end{document}